\documentclass[11pt]{amsart}
\usepackage{amssymb,amsmath}

\theoremstyle{plain}
\newtheorem{thrm}{Theorem}[section]
\newtheorem{lemma}[thrm]{Lemma}
\newtheorem{prop}[thrm]{Proposition}
\newtheorem{cor}[thrm]{Corollary}
\newtheorem{rmrk}[thrm]{Remark}
\newtheorem{dfn}[thrm]{Definition}

\numberwithin{equation}{section}

\begin{document}
\newcommand{\SL}{\mathcal L^{1,p}( D)}
\newcommand{\Lp}{L^p( Dega)}
\newcommand{\CO}{C^\infty_0( \Omega)}
\newcommand{\Rn}{\mathbb R^n}
\newcommand{\Rm}{\mathbb R^m}
\newcommand{\R}{\mathbb R}
\newcommand{\Om}{\Omega}
\newcommand{\Hn}{\mathbb H^n}
\newcommand{\N}{\mathbb N}
\newcommand{\aB}{\alpha B}
\newcommand{\eps}{\epsilon}
\newcommand{\BVX}{BV_X(\Omega)}
\newcommand{\p}{\partial}
\newcommand{\IO}{\int_\Omega}
\newcommand{\bG}{\mathbb{G}}
\newcommand{\bg}{\mathfrak g}
\newcommand{\Bux}{\mbox{Box}}
\newcommand{\al}{\alpha}
\newcommand{\til}{\tilde}
\newcommand{\nuX}{\boldsymbol{\nu}^X}
\newcommand{\bN}{\boldsymbol{N}}
\newcommand{\nh}{\nabla^H}
\newcommand{\Gp}{G_{D,p}}
\newcommand{\n}{\boldsymbol \nu}


\title[Mutual absolute continuity of harmonic and surface measures, etc.]{Mutual absolute continuity of harmonic and surface measures for H\"ormander type operators}

\dedicatory{Dedicated to Professor Maz'ya, on his 70th birthday}

\author{Luca Capogna}
\address{Department of Mathematics\\University of Arkansas\\Fayetteville, AR 72701}
\email[Luca Capogna]{lcapogna@uark.edu}
\thanks{First author supported in part by NSF Career grant DMS-0134318}
\author{Nicola Garofalo}
\address{Department of Mathematics\\
Purdue University \\
West Lafayette IN 47907-1968}
\email[Nicola Garofalo]{garofalo@math.purdue.edu}
\thanks{Second author supported in part by NSF Grant DMS-07010001}
\author{Duy-Minh Nhieu}
\address{Department of Mathematics \\
San Diego Christian College \\
2100 Greenfield dr\\
El Cajon CA 92019} \email[Duy-Minh Nhieu]{dnhieu@sdcc.edu}

\date{\today}

%
%

\maketitle


\section{\textbf{Introduction}}

\vskip 0.2in

In this paper we study  the Dirichlet problem for the sub-Laplacian
associated with a system $X=\{X_1,...,X_m\}$ of $C^\infty$ real
vector fields in $\Rn$ satisfying H\"ormander's finite rank
condition
\begin{equation}\label{frc}
rank\ Lie[X_1,...,X_m]\ \equiv\ n .
\end{equation}

Throughout this paper $n\geq 3$, and $X_j^*$ denotes the formal
adjoint of $X_j$. The sub-Laplacian associated with $X$ is defined
by
\begin{equation}\label{sublap}
\mathcal{L} u\ =\ \sum_{j=1}^m\ X_j^*X_j u\ .
\end{equation}

A distributional solution of $\mathcal{L}u=0$  is called
$\mathcal{L}$-\emph{harmonic}. H\"ormander's hypoellipticity theorem
\cite{H} guarantees that every $\mathcal{L}$-harmonic function is
$C^\infty$, hence it is a classical solution of $\mathcal{L}u=0$. We
consider a bounded open set $D\subset \Rn$, and study the Dirichlet
problem
\begin{equation}\label{DP}
\begin{cases}
\mathcal{L} u = 0 \quad\text{in} \ D\ ,\\
u=\phi\quad\text{on }\partial{D}\ .
\end{cases}
\end{equation}

Using Bony's maximum principle \cite{B} one can show that for any
$\phi\in C(\partial D)$ there exists a unique Perron-Wiener-Brelot
solution $H^D_\phi$ to \eqref{DP}. We focus on the boundary
regularity of the solution. In particular, we identify a class of
domains, which are referred to as $ADP_X$ domains (\emph{admissible
for the Dirichlet problem}), for which we prove the mutual absolute
continuity of the $\mathcal L$-harmonic measure $d\omega^x$ and of
the so-called horizontal perimeter measure $d\sigma_X =
P_X(D;\cdot)$ on $\partial D$. The latter constitutes the
appropriate replacement for the standard surface measure on $\p D$
and plays a central role in sub-Riemannian geometry. Moreover, we
show that a reverse H\"older inequality holds for a suitable Poisson
kernel which is naturally associated with the system $X$. As a
consequence of such reverse H\"older inequality we then derive the
solvability of \eqref{DP} for boundary data $\phi\in L^p(\partial D,
d\sigma_X)$, for $1<p\leq \infty$. If instead the domain $D$ belongs
to the smaller class $\sigma-ADP_X$ introduced in Definition
\ref{D:sADPX} below, we prove that $\mathcal L$-harmonic measure is
mutually absolutely continuous with respect to the standard surface
measure, and we are able to solve the Dirichlet problem for
\eqref{sublap} for boundary data $\phi\in L^p(\partial D, d\sigma)$,
for $1<p\leq \infty$.

The connection between harmonic and surface measure is a central
question in the study of boundary value problems for second order
partial differential equations. As it is well-known a basic result
of Brelot allows to solve the Dirichlet problem for the standard
Laplacian when the boundary datum is in $L^1$ with respect to the
harmonic measure. However, since the latter is difficult to pin
down, it becomes important to know for what domains one can solve
the Dirichlet problem when the boundary data are in some $L^p$ space
with respect to the ordinary surface measure $d\sigma$. In his
ground-breaking 1977 paper \cite{Da1} Dahlberg was able to settle
the long standing conjecture that in a Lipschitz domain in $\Rn$
harmonic measure for the Laplacian and Hausdorff measure $H^{n-1}$
restricted to the boundary are mutually absolutely continuous. One
should also see the sequel paper \cite{Da2} where the mutual
absolute continuity was obtained as a consequence of the reverse
H\"older inequality for the kernel function $k = d\omega/d\sigma$.
For $C^1$ domains Dahlberg's result was also independently proved by
Fabes, Jodeit and Rivi\`ere \cite{FJR} by the method of layer
potentials.

The results in this paper should be considered as a subelliptic
counterpart of Dahlberg's results in \cite{Da1}, \cite{Da2}. There
are however four aspects which substantially differ from the
analysis of the ordinary Laplacian, and they are all connected with
the presence of the so-called characteristic points on the boundary.
In order to describe these aspects we recall that given a $C^1$
domain $D\subset \Rn$, a point $x_o\in \p D$ is called
\emph{characteristic} for the system $X=\{X_1,...,X_m\}$ if
indicating with $\bN(x_o)$ a normal vector to $\p D$ in $x_o$ one
has
\[ <\bN(x_o),X_1(x_o)>\ =\ ... \ =\ <\bN(x_o),X_m(x_o)>\ =\ 0\ . \]

The characteristic set of $D$, hereafter denoted by $\Sigma =
\Sigma_{D,X}$, is the collection of all characteristic points of $\p
D$. It is a closed subset of $\p D$, and it is compact if $D$ is
bounded. We next introduce the most important prototype of a
sub-Riemannian space: the Heisenberg group $\Hn$. This is the
stratified nilpotent Lie group of step two whose underlying manifold
is $\mathbb C^n \times \mathbb R$ with group law $(z,t)\circ (z't')
= (z+z', t+t' - \frac{1}{2}\ Im( z\cdot \overline{z'}))$. If $x =
(x_1,...,x_n), y = (y_1,...,y_n)$, and we identify $z = x + i y\in
\mathbb C^n$ with the vector $(x,y) \in \mathbb R^{2n}$, then in the
real coordinates $(x,y,t)\in \R^{2n+1}$ a basis for the Lie algebra
of left-invariant vector fields on $\Hn$ is given by the vector
fields
\begin{equation}\label{vf}
X_j\ =\ \frac{\partial}{\partial x_j}\ -\ \frac{y_j}{2}\
\frac{\partial}{\partial t},\quad\quad  X_{n+j}\ =\
\frac{\partial}{\partial y_{j}}\ +\ \frac{x_j}{2}\
\frac{\partial}{\partial t},\quad j=1,...,n,\ \ \ \ \frac{\p }{\p
t}\ .
\end{equation}

In view of the commutation relations \[ [X_j, X_{n+k}]\ =\
\delta_{jk}\ \frac{\partial}{\partial t}\ ,\ \ \  j, k = 1,...,n\ ,
\]
the system $X = \{X_1,...,X_{2n}\}$ generates the Lie algebra of
$\Hn$. The real part of the Kohn-Spencer sub-Laplacian on $\Hn$ is
given by
\begin{equation}\label{kohn}
\mathcal L_o\ =\ \sum_{j=1}^{2n} X_j^2\ =\ \Delta_z +
\frac{|z|^2}{4} D_{tt} + D_t (\sum_{j=1}^n x_j D_{y_j} - y_j
D_{x_j})\ . \end{equation}

This remarkable operator plays an ubiquitous role in several
branches of mathematics and of the applied sciences. We stress that
$\mathcal L_o$ fails to be elliptic at every point. Concerning the
distinctions mentioned above we note:

\medskip

1) Differently from the classical case, in the subelliptic Dirichlet
problem \eqref{DP} the Euclidean smoothness of the ground domain is
of no significance from the standpoint of the intrinsic geometry
near the characteristic set $\Sigma$. In this geometry even a domain
with real analytic boundary looks like a cuspidal domain near one of
its characteristic points. Since bounded domains typically have
non-empty characteristic set it follows that the notion of
``Lipschitz domain" is not as important as in the Euclidean setting,
and one has to abandon it in favor of a more general one based on
purely metrical properties, see \cite{CG}. With these comments in
mind, in this paper we will assume that the domain $D$ in \eqref{DP}
be $NTA_X$ (\emph{non-tangentially accessible} with respect to the
Carnot-Carath\'eodory distance associated with the system $X$, see
Definition \ref{D:NTA} below) and $C^\infty$. The former property
allows us to use some fundamental results developed in \cite{CG},
whereas the smoothness assumption permits to use tools from calculus
away from the characteristic set. In this connection we mention that
the $C^\infty$ hypothesis guarantees, in view of the results in
\cite{KN1}, that the Green function for \eqref{DP} and singularity
at a given point in $D$ is smooth up to the boundary away from
$\Sigma$, see Theorem \ref{T:KN} below.

2) Another striking phenomenon is that in the subelliptic Dirichlet
problem nonnegative $\mathcal L$-harmonic functions which vanish on
a portion of the boundary can do so at very different rates. The
dual aspect of this phenomenon is that nonnegative $\mathcal
L$-harmonic functions which blow up at the boundary (such as for
instance the Poisson kernel) have very different rates of blow-up
depending on whether the limit point is characteristic or not, see
\cite{GV}. This is in sharp contrast with the classical setting. It
is well-known \cite{G} that in a $C^{1,1}$ domain all nonnegative
harmonic functions (or solutions to more general elliptic and
parabolic equations) vanishing on a portion of the boundary must
vanish exactly like the distance to the boundary itself. This fails
miserably in the subelliptic setting because of characteristic
points on the boundary. For instance, in $\Hn$ the so-called gauge
ball \[ B\ =\ \{(z,t)\in \Hn\mid |z|^4 + 16 t^2 < 1\}\ \] is a real
analytic domain with two isolated characteristic points $P^\pm =
(0,\pm \frac{1}{4})$. With $\mathcal L_o$ defined by \eqref{kohn},
the function $u(z,t) = t+\frac{1}{4}$ is a nonnegative $\mathcal
L_o$-harmonic function in $B$ which along the $t$-axis vanishes at
the (characteristic) boundary point $P^- = (0,-\frac{1}{4})$ as the
square of the Carnot-Carath\'eodory distance to $P^-$. On the other
hand, the function $u(z,t) = x_1 +1$ is a nonnegative $\mathcal
L_o$-harmonic function in $B$ which along the $x_1$-axis vanishes at
the (non-characteristic) boundary point $P_1 = (-e_1,0)$, where $e_1
= (1,0,...,0)\in \R^{2n}$, like the distance to $P_1$. Thus, there
is not one single rate of vanishing for nonnegative $\mathcal
L_o$-harmonic functions in smooth domains in $\Hn$! Despite this
negative phenomenon in \cite{CG} two of us proved that in a $NTA_X$
domain all nonnegative $\mathcal L$-harmonic functions vanishing on
a portion of the boundary (characteristic or not) must do so at the
same rate. This result, known as the \emph{comparison theorem},
plays a fundamental role in the present paper. Returning to the
above example of the gauge ball $B\subset \Hn$, the comparison
theorem implies in particular that all nonnegative solutions of
$\mathcal L_o u = 0$ which vanish in a boundary neighborhood of the
point $P^- = (0,-\frac{1}{4})$, must vanish non-tangentially like
the square of the distance to the boundary (and not linearly like in
the classical case)!

3) The third aspect which we want to emphasize is closely connected
with the discussion in 1) and leads us to introduce the third main
assumption in the present paper paper. In \cite{J1}, \cite{J2} D.
Jerison studied the Dirichlet problem \eqref{DP} near characteristic
points for $\mathcal L_o$. He proved in \cite{J1} that for a
$C^\infty$ domain $D \subset \Hn$ if the datum $\phi$ belongs to a
Folland-Stein H\"older class $\Gamma^\beta$, then $H^D_\phi$ is in
$\Gamma^\alpha(\overline D)$, for some $\alpha$ depending on $\beta$
and on the domain $D$. It was also shown in \cite{J1} that, given
any $\alpha\in (0,1)$ there exists $M = M(\alpha)>0$ for which the
real analytic domain
\[
\Om_M\ =\ \{(z,t)\in \Hn\mid t > - M |z|^2\}\ , \] admits a
$\mathcal L_o$-harmonic function $u$ such that $u = 0$ on $\p \Om_M$
and which belongs exactly to the H\"older class $\Gamma^{\alpha}$
(in the sense that it is not any smoother) in any neighborhood of
the (characteristic) boundary point $e = (0,0)$. Once again, this
example shows that, despite the (Euclidean) smoothness of the domain
and of the boundary datum, near a characteristic point the domain
appears quite non-smooth with respect to the intrinsic geometry of
the vector fields $X_1,...,X_{2n}$. In fact, since the paraboloid
$\Om$ is a scale invariant region with respect to the non-isotropic
group dilations $(z,t)\to (\lambda z, \lambda^2 t)$, the smooth
domain $\Om_M$ should be thought of as a non-convex cone from the
point of view of the intrinsic geometry of $\mathcal L_o$ (for a
discussion of Jerison's example see section \ref{S:Jerison}). This
suggests that by imposing a condition similar to the classical
Poincar\'e tangent outer sphere \cite{P} one should be able to rule
out Jerison's negative example and possibly control the intrinsic
gradient $XG$ of the Green function near the characteristic set.
This intuition was proved successful in the papers \cite{LU},
\cite{CGN1}, which were respectively concerned with the Heisenberg
group and with Carnot groups of Heisenberg type. In this paper we
generalize this idea and prove the boundedness of the Poisson kernel
in a neighborhood of the boundary under the hypothesis that the
domain $D$ in \eqref{DP} satisfy what we call a \emph{tangent outer
$X$-ball condition}. It is worth emphasizing that the $X$-balls in
our definition are not metric balls, but instead they are the
(smooth) level sets of the fundamental solution of the sub-Laplacian
$\mathcal L$. The metric balls are not smooth (see \cite{CG}) and
therefore it would not be possible to have a notion of tangency
based on these sets.

4) In Dahlberg's mentioned theorem on the mutual absolute continuity
between harmonic and surface measure in a Lipschitz domain $D\subset
\Rn$ there is one important property which, although confined to the
background, plays a central role. If we denote by $\sigma =
H^{n-1}|_{\p D}$ the surface measure on the boundary, then there
exists constants $\alpha, \beta>0$ depending on $n$ and on the
Lipschitz character of $D$ such that
\[
\alpha\ r^{n-1}\ \leq\ \sigma(\p D \cap B(x,r))\ \leq\ \beta\
r^{n-1}\ ,
\]
for any $x\in \p D$ and any $r>0$. A property like this is referred
to as the $1$-Ahlfors regularity of $\sigma$, and thanks to it
surface measure is the natural measure on $\p D$. Things are quite
different in the subelliptic Dirichlet problem. Consider in fact the
gauge ball $B$ as in 2), with its two (isolated) characteristic
points $P^\pm = (0, \pm \frac{1}{4})$ of $\p B$. Simple calculations
show that denoting by $B(P^\pm,r)$ a gauge ball centered at one of
the points $P^\pm$ with radius $r$, then one has for small $r>0$
\begin{equation}\label{hp}
\sigma(\p B \cap B(P^\pm,r))\ \cong\ r^{Q-2}\ ,
\end{equation}
where $Q = 2n + 2$ is the so-called homogeneous dimension of $\Hn$
relative to the non-isotropic dilations $(z,t) \to (\lambda z ,
\lambda^2 t)$ associated with the grading of the Lie algebra of
$\Hn$. The latter equation shows that at the characteristic points
$P^\pm$ surface measure becomes quite singular and it does not scale
correctly with respect to the non-isotropic group dilations. The
appropriate ``surface measure" in sub-Riemannian geometry is instead
the so-called horizontal perimeter $P_X(D;\cdot)$ introduced in
\cite{CDG2} which on surface metric balls is defined in the
following way
\[ \sigma_X(\p D \cap B_d(x,r))\ \overset{def}{=}\ P_X(D;B_d(x,r))\ .
\]

To motivate such appropriateness we recall that it was proved in
\cite{DGN}, \cite{DGN2} that for every $C^2$ bounded domain
$D\subset \Hn$ one has for every $x\in \p D$ and every $0<r<R_o(D)$
\[
\alpha\ r^{Q-1}\ \leq\ \sigma_X(\p D \cap B_d(x,r))\ \leq\ \beta\
r^{Q-1}\ .
\]

Now it was also shown in these papers that the inequality in the
right-hand side alone suffices to establish the existence of the
traces of Sobolev functions on the boundary. Remarkably, as we prove
in Theorem \ref{T:RHGeneral} below, such a one-sided Ahlfors
property also suffices to establish the mutual absolute continuity
of $\mathcal L$-harmonic and horizontal perimeter measure. Such
property will constitute the last basic assumption of our results,
to which we finally turn. We need to introduce the relevant class of
domains.

\begin{dfn}\label{D:ADPX}
Given a system $X = \{X_1,...,X_m\}$ of smooth vector fields
satisfying \eqref{frc}, we say that a connected bounded open set $
D\subset \Rn$ is \emph{admissible for the Dirichlet problem
\eqref{DP}} with respect to the system $X$, or simply $ADP_X$, if:

i) $D$ is of class $C^\infty$;

ii) $D$ is non-tangentially accessible ($NTA_X$) with respect to the
Carnot-Caratheodory metric associated to the system
$\{X_1,...,X_m\}$ (see Definition \ref{D:NTA});

iii) $D$ satisfies a uniform tangent outer $X$-ball condition (see
Definition \ref{D:OB});

iv) The horizontal perimeter measure is upper $1$-Ahlfors regular.
This means that there exist $A, R_o>0$ depending on $X$ and $D$ such
that for every $x\in \p D$ and $0<r<R_o$ one has
\[
\sigma_X(\p D \cap B_d(x,r))\ \leq\ A\ \frac{|B_d(x,r)|}{r}\ .
\]
\end{dfn}

The constants appearing in $iv)$ and in Definitions \ref{D:OB} and
\ref{D:NTA} will be referred to as the $ADP_X$-parameters of $D$. We
introduce next a central character in this play, the subelliptic
Poisson kernel of $D$. In fact, we define two such functions, each
one playing a different role. Let $G(x,y) = G_D(x,y) = G(y,x)$
indicate the Green function for the sub-Laplacian \eqref{sublap} and
for an $ADP_X$ domain $D$ \footnote{In \cite{B} it was proved that
any bounded open set admits a Green function}. By H\"ormander's
theorem \cite{H} and the results in \cite{KN1}, see Theorem
\ref{T:KN} below, for any fixed $x\in D$ the function $y\to G(x,y)$
is $C^\infty$ up to the boundary in a suitably small neighborhood of
any non-characteristic point $y_o\in \p D$. Let $\n(y)$ indicate the
outer unit normal in $y\in \p D$. At every point $y\in \p D$ we
denote by $\bN^X(y)$ the vector defined by
\[
\bN^X(y)\ =\ \left(<\n(y),X_1(y)>,...,<\n(y),X_m(y)>\right)\ . \]

We also set
\[ W(y)\ =\ |\bN^X(y)|\ =\ \sqrt{<\n(y),X_1(y)>^2 + ... + <\n(y),X_m(y)>^2} \ .
\]

We note explicitly that it was proved in \cite{CDG2} that on $\p D$
\[
d\sigma_X\ =\ W\ d\sigma\ . \]

Denoting with $\Sigma$ the characteristic set of $D$, we remark that
the vector $\bN^X(y) = 0$ if and only if $y\in \Sigma$. For $y\in \p
D \setminus \Sigma$ we define the \emph{horizontal Gauss map} at $y$
by letting
\[
\nuX(y)\ =\ \frac{\bN^X(y)}{|\bN^X(y)|}\ .
\]

\begin{dfn}\label{D:poisson}
Given a $C^\infty$ bounded open set $D\subset \Rn$, for every
$(x,y)\in D\times (\p D\setminus \Sigma)$ we define \emph{the
subelliptic Poisson kernels} as follows
\[
P(x,y)\ =\ <XG(x,y),\bN^X(y)>\ ,\ \ K(x,y)\ =\ \frac{P(x,y)}{W(y)}\
=\ <XG(x,y),\n^X(y)>\ .
\]
\end{dfn}

We emphasize here that the reason for which in the definition of
$P(x,y)$ and $K(x,y)$ we restrict $y$ to $\p D\setminus \Sigma$ is
that, as we have explained in 3) above (see also section
\ref{S:Jerison}), the horizontal gradient $XG(x,y)$ may not be
defined at points of $\Sigma$. Since as we have observed the
function $W$ vanishes on $\Sigma$, it should be clear that the
function $K(x,y)$ is more singular then $P(x,y)$ at the
characteristic points. However, such additional singularity is
balanced by the fact that the density $W$ of the measure $\sigma_X$
with respect to surface measure vanishes at the characteristic
points. As a consequence, $K(x,y)$ is the appropriate subelliptic
Poisson kernel with respect to the intrinsic measure $\sigma_X$,
whereas $P(x,y)$ is more naturally attached to the ``wrong measure"
$\sigma$.

Hereafter, for $x\in \p D$ it will be convenient to indicate with
$\Delta(x,r) = \p D \cap B_d(x,r)$, the boundary metric ball
centered at $x$ with radius $r>0$. The first main result in this
paper is contained the following theorem.

\begin{thrm}\label{T:RHGeneral}
Let $D \subset \Rn$ be a $ADP_X$ domain. For every $p>1$ and any
fixed $x_1\in D$ there exist positive constants $C, R_1$, depending
on $p, M, R_o, x_1$, and on the $ADP_X$ parameters, such that for
$x_o\in \partial{D}$ and $0<r<R_1$ one has
\begin{equation*}
\left(\frac{1}{\sigma_X(\Delta(x_o,r))} \ \int_{\Delta(x_o,r)}
K(x_1,y)^p d\sigma_X(y) \right)^{\frac{1}{p}} \ \leq C\
\frac{1}{\sigma_X(\Delta(x_o,r))} \ \int_{\Delta(x_o,r)} K(x_1,y)
d\sigma_X(y)\ .
\end{equation*}
Moreover, the measures $d\omega^{x_1}$ and $d\sigma_X$ are mutually
absolutely continuous.
\end{thrm}

By combining Theorem \ref{T:RHGeneral} with the results if \cite{CG}
we can solve the Dirichlet problem for boundary data in $L^p$ with
respect to the perimeter measure $d\sigma_X$. To state the relevant
results we need to introduce a definition. Given $D$ as in Theorem
\ref{T:RHGeneral}, for any $y\in\partial D$ and $\alpha > 0$ a
nontangential region at $y$ is defined by
\[
\Gamma_\alpha(y)\ =\ \{x\in D\,|\, d(x,y) \leq
(1+\alpha)d(x,\partial D)\}\ .
\]

Given a function $u\in C(D)$, the $\alpha$-nontangential maximal
function of $u$ at $y$ is defined by
\[
N_\alpha(u)(y)\ =\ \underset{x\in\Gamma_\alpha(y)}{sup}\,|u(x)|\ .
\]

\begin{thrm}\label{T:Dirichlet}
Let $D\subset\Rn$ be a $ADP_X$ domain. For every $p>1$ there exists
a constant $C>0$ depending on $D, X$ and $p$ such that if $f\in
L^p(\partial D,d\sigma_X)$, then
\[
H^{D}_{f}(x)\ =\ \int_{\partial{D}}\ f(y)\ K(x,y)\ d\sigma_X(y)\ ,
\]
and
\[
\|N_\alpha(H^D_f)\|_{L^p(\partial D,d\sigma_X)}\ \leq\ C\
\|f\|_{L^p(\partial D,d\sigma_X)}\ .
\]
Furthermore, $H^D_f$ converges nontangentially $\sigma_X$-a.e. to
$f$ on $\partial D$.

\end{thrm}

Theorems \ref{T:RHGeneral} and \ref{T:Dirichlet} constitute
appropriate sub-elliptic versions of Dahlberg's mentioned results in
\cite{Da1}, \cite{Da2}. These theorems generalize those in
\cite{CGN2} relative to Carnot groups of Heisenberg type. We mention
at this point that, as we prove in Theorem \ref{T:ndeg} below, for
any $C^{1,1}$ domain $D\subset \Rn$ which is $NTA_X$ the horizontal
perimeter measure is lower $1$-Ahlfors (this is a basic consequence
of the isoperimetric inequality in \cite{GN1}). Combining this
result with the assumption $iv)$ in Definition \ref{D:ADPX}, we
conclude that for any $ADP_X$ domain the measure $\sigma_X$ is
$1$-Ahlfors. In particular, $\sigma_X$ is also doubling, see
Corollary \ref{C:ahlfors}. This information plays a crucial role in
the proof of Theorem \ref{T:Dirichlet}.

On the other hand, even if the ordinary surface measure $\sigma$ is
the ``wrong one" in the subelliptic Dirichlet problem, it would
still be highly desirable to know if there exist situations in which
\eqref{DP} can be solved for boundary data in some $L^p$ with
respect to $d\sigma$. To address this question in Definition
\ref{D:sADPX} we introduce the class of $\sigma-ADP_X$ domains. The
latter differs from that of $ADP_X$ domains for the fact that the
assumption $iv)$ is replaced by the following
\emph{balanced-degeneracy} assumption on $\sigma$: there exist $B,
R_o>0$ depending on $X$ and $D$ such that for every $x_o\in \p D$
and $0<r<R_o$ one has
\[
\left(\underset{y\in\Delta(x_o,r)}{max}\;W(y)\right)\,\sigma(\Delta(x_o,r))\
\leq\ B\ \frac{|B_d(x_o,r)|}{r}\ .
\]

As we have previously observed surface measure becomes singular near
a characteristic point. On the other hand, the angle function $W$
vanishes, thus balancing the singularities of $\sigma$. For
$\sigma-ADP_X$-domains we obtain the following two results which
respectively establish the mutual absolute continuity of $\mathcal
L$-harmonic and surface measure $d\sigma$, and the solvability of
the Dirichlet problem with data in $L^p(\p D,d\sigma)$.

\begin{thrm}\label{T:RHGeneral-P}
Let $D \subset \Rn$ be a  $\sigma-ADP_X$ domain. Fix $x_1 \in D$.
For every $p>1$ there exist positive constants $C, R_1$, depending
on $p, M, R_o, x_1$, and on the $\sigma-ADP_X$ parameters, such that
for every $y\in
\partial{D}$ and $0<r<R_1$ one has
\begin{equation*}
\left(\frac{1}{\sigma(\Delta(x_o,r))} \ \int_{\Delta(x_o,r)}
P(x_1,y)^p d\sigma(y) \right)^{\frac{1}{p}} \ \leq C\
\frac{1}{\sigma(\Delta(x_o,r))} \ \int_{\Delta(x_o,r)} P(x_1,y)
d\sigma(y)\ .
\end{equation*}
Moreover, the measures $d\omega^x$ and $d\sigma$ are mutually
absolutely continuous.
\end{thrm}

We mention explicitly that a basic consequence of Theorem
\ref{T:RHGeneral-P} is that the standard surface measure on the
boundary of a $\sigma-ADP_X$ domain is doubling.

\begin{thrm}\label{T:Dirichletsigma}
Let $D\subset\Rn$ be a $\sigma-ADP_X$ domain. For every $p>1$ there
exists a constant $C>0$ depending on $D, X$ and $p$ such that if
$f\in L^p(\partial D,d\sigma)$, then
\[
H^{D}_{f}(x)\ =\ \int_{\partial{D}}\ f(y)\ P(x,y)\ d\sigma(y)\ ,
\]
and
\[
\|N_\alpha(H^D_f)\|_{L^p(\partial D,d\sigma)}\ \leq\ C\
\|f\|_{L^p(\partial D,d\sigma)}.
\]
Furthermore, $H^D_f$ converges nontangentially $\sigma$-a.e. to $f$
on $\partial D$.
\end{thrm}

Concerning Theorems \ref{T:RHGeneral}, \ref{T:Dirichlet},
\ref{T:RHGeneral-P} and \ref{T:Dirichletsigma} we mention that large
classes of domains to which they apply were found in \cite{CGN2},
but one should also see \cite{LU} for domains satisfying assumption
$iii)$ in Definition \ref{D:ADPX}. The discussion of these examples
is taken up in section \ref{S:ex}.

In closing we briefly describe the organization of the paper. In
section \ref{S:Prelim} we collect some known results on
Carnot-Carath\'eodory metrics which are needed in the paper. In
section \ref{S:DP} we discuss some known results on the subelliptic
Dirichlet problem which constitute the potential theoretic backbone
of the paper. In section \ref{S:Jerison} we discuss Jerison's
mentioned example.

Section \ref{S:SE} is devoted to proving some new interior a priori
estimates of Cauchy-Schauder type. Such estimates are obtained by
means of a family of subelliptic mollifiers which were introduced by
Danielli and two of us in \cite{CDG1}, see also \cite{CDG2}. The
main results are Theorems \ref{T:Harmonic}, \ref{T:AprioriE}, and
Corollary \ref{C:Harnack}. We feel that, besides being instrumental
to the present paper, these results will prove quite useful in
future research on the subject.

In section \ref{S:bdry} we use the interior estimates in Theorem
\ref{T:Harmonic}  to prove that if a domain satisfies a uniform
outer tangent $X$-ball condition, then the horizontal gradient of
the Green function $G$ is bounded up to the boundary, hence, in
particular, near $\Sigma$, see Theorem \ref{T:XG}. The proof of such
result rests in an essential way on the linear growth estimate
provided by Theorem \ref{T:Growth}. Another crucial ingredient is
Lemma \ref{L:Gamma} which allows a delicate control of some ad-hoc
subelliptic barriers. In the final part of the section we show that,
by requesting the uniform outer $X$-ball condition only in a
neighborhood of the characteristic set $\Sigma$, we are still able
to obtain the boundedness of the horizontal gradient of $G$ up to
the characteristic set, although we now loose the uniformity in the
estimates, see Theorem \ref{T:GreenLocal}, \ref{T:XGLocal} and
Corollary \ref{C:LipschitzSigma}.

In section \ref{S:PK} we establish a Poisson type representation
formula for domains which satisfy the uniform outer $X$-ball
condition in a neighborhood of the characteristic set. This result
generalizes a similar Poisson type formula in the Heisenberg group
$\Hn$ obtained by Lanconelli and Uguzzoni in \cite{LU}, and extended
in \cite{CGN2} to Carnot groups of Heisenberg type. If generically
the Green function of a smooth domain had bounded horizontal
gradient \emph{up to the characteristic set}, then such Poisson
formula would follow in an elementary way from integration by parts.
As we previously stressed, however, things are not so simple and the
boundedness of $XG$ fails in general near the characteristic set.
However, when $D\subset \Rn$ satisfies the uniform outer $X$-ball
condition in a neighborhood of the characteristic set, then
combining Theorem \ref{T:XG} with the estimate \[ K(x,y)\ \leq\
|XG(x,y)|\ , \ \ x \in D, y\in \p D\ ,
\] see \eqref{i6}, we prove the boundedness of the Poisson kernel
$y\to K(x,y)$ on $\p D$. The main result in section \ref{S:PK} is
Theorem \ref{T:HM}. This representation formula with the estimates
of the Green function in sections \ref{S:SE} and \ref{S:bdry} lead
to a priori estimates in $L^p$ for the solution to \eqref{DP} when
the datum $\phi\in C(\partial D)$. Solvability of \eqref{DP} with
data in Lebesgue classes requires, however, a much deeper analysis.

The first observation is that the outer ball condition alone does
not guarantee the development of a rich potential theory. For
instance, it may not be possible to find: a) Good nontangential
regions of approach to the boundary from within the domain; b)
Appropriate interior Harnack chains of nontangential balls. This is
where the basic results on $NTA_X$ domains from \cite{CG} enter the
picture. In the opening of section \ref{S:RHI} we recall the
definition of $NTA_X$-domain along with those results from \cite{CG}
which constitute the foundations of the present study. Using these
results we establish Theorem \ref{T:Komega}. The remaining part of
the section is devoted to proving Theorems \ref{T:RHGeneral},
\ref{T:Dirichlet}, \ref{T:RHGeneral-P} and \ref{T:Dirichletsigma}.

Finally, section \ref{S:ex} is devoted to the discussion of examples
of $ADP_X$ and $\sigma-ADP_X$ domains and of some open problems.

\vskip 0.6in

\section{\textbf{Preliminaries}}\label{S:Prelim}

\vskip 0.2in

In $\Rn$, with $n\geq 3$, we consider a system  $X =
\{X_1,...,X_m\}$ of $C^\infty$ vector fields satisfying
H\"ormander's finite rank condition \eqref{frc}. A piecewise $C^1$
curve $\gamma:[0,T]\to \Rn$ is called \emph{sub-unitary} \cite{FP}
if whenever $\gamma'(t)$ exists one has for every $\xi\in\Rn$
\[
<\gamma'(t),\xi>^2\ \leq\ \sum_{j=1}^m <X_j(\gamma(t)),\xi>^2 .
\]

We note explicitly that the above inequality forces $\gamma '(t)$ to
belong to the span of $\{X_1(\gamma (t)),...,$ $ X_m(\gamma (t))\}$.
The sub-unit length of $\gamma$ is by definition $l_s(\gamma)=T$.
Given $x, y\in \Rn$, denote by $\mathcal S_\Om(x,y)$ the collection
of all sub-unitary $\gamma:[0,T]\to \Om$ which join $x$ to $y$. The
accessibility theorem of Chow and Rashevsky, \cite{Ra}, \cite{Chow},
states that, given a connected open set $\Om\subset \Rn$, for every
$x,y\in \Om$ there exists $\gamma \in \mathcal S_\Om(x,y)$. As a
consequence, if we pose
\[
d_{\Om}(x,y)\ =\ \text{inf}\ \{l_s(\gamma)\mid \gamma \in \mathcal S_\Om(x,y)\} ,
\]
we obtain a distance on $\Om$, called the \emph{Carnot-Carath\'eodory distance on $\Om$}, associated with the system $X$. When $\Om = \Rn$, we write $d(x,y)$ instead of $d_{\Rn}(x,y)$. It is clear that $d(x,y) \leq d_\Om(x,y)$, $x, y\in \Om$, for every connected open set $\Om \subset \Rn$. In \cite{NSW} it was proved that for every connected $\Om \subset \subset \Rn$ there exist $C, \epsilon >0$ such that
\begin{equation}\label{CCeucl}
C\ |x - y|\ \leq d_\Om(x,y)\ \leq C^{-1}\ |x - y|^\epsilon , \quad\quad\quad x, y \in \Om .
\end{equation}

This gives $d(x,y)\ \leq C^{-1} |x - y|^\epsilon$, $x, y\in \Om$, and therefore
\[
i: (\Rn, |\cdot|)\to (\Rn, d) \quad\quad\quad is\,\ continuous .
\]

It is easy to see that also the continuity of the opposite inclusion holds \cite{GN1}, hence the metric and the Euclidean topology are compatible.

For $x\in \Rn$ and $r>0$, we let $B_d(x,r)\ =\ \{y\in \Rn\mid d(x,y) < r \}$.
The basic properties of these balls were established by Nagel, Stein and Wainger in their seminal paper \cite{NSW}. Denote by $Y_1,...,Y_l$ the collection of the $X_j$'s and of those commutators which are needed to generate $\Rn$. A formal ``degree" is assigned to each $Y_i$, namely the corresponding order of the commutator. If $I = (i_1,...,i_n), 1\leq i_j\leq l$ is a $n$-tuple of integers, following \cite{NSW} we let $d(I) = \sum_{j=1}^n deg(Y_{i_j})$, and $a_I(x) = \text{det}\ (Y_{i_1},...,Y_{i_n})$. The \emph{Nagel-Stein-Wainger polynomial} is defined by
\begin{equation}\label{pol}
\Lambda(x,r)\ =\ \sum_I\ |a_I(x)|\ r^{d(I)}, \quad\quad\quad\quad r > 0.
\end{equation}

For a given bounded open set $U\subset \Rn$, we let
\begin{equation}\label{Q}
Q\ =\ \text{sup}\ \{d(I)\mid \ |a_I(x)| \ne 0, x\in U\},\quad\quad Q(x)\ =\ \text{inf}\ \{d(I)\mid |a_I(x)| \ne 0\} ,
\end{equation}
and notice that $n\leq Q(x)\leq Q$. The numbers $Q$ and $Q(x)$ are respectively called the \emph{local homogeneous dimension} of $U$ and the homogeneous dimension at $x$ with respect to the system $X$.

\medskip
\begin{thrm}[\textbf{\cite{NSW}}]\label{T:db}
For every bounded set $U\subset\Rn$, there exist
constants $C, R_o>0$ such that, for any $x\in U$, and $0 < r \leq R_o$,
\begin{equation}\label{nsw2}
C\ \Lambda(x,r)\ \leq\ |B_d(x,r)|\ \leq\  C^{-1}\ \Lambda(x,r).
\end{equation}
As a consequence, one has with $C_1 = 2^Q$
\begin{equation}\label{dc}
|B_d(x,2r)|\ \leq\ C_1\ |B_d(x,r)| \qquad\text{for every}\quad x\in U\quad\text{and}\quad 0< r \leq R_o.
\end{equation}
\end{thrm}

\medskip

The numbers $C_1, R_o$ in \eqref{dc} will be referred to as the \emph{characteristic local parameters} of $U$. Because of (2.2), if we let
\begin{equation}\label{E}
E(x,r)\ =\ \frac{\Lambda(x,r)}{r^{2}} ,
\end{equation}
then the function $r \to E(x,r)$ is strictly increasing. We denote
by $F(x,\cdot)$ the  inverse function of $E(x,\cdot)$, so that
$F(x,E(x,r)) = r$. Let $\Gamma(x,y)=\Gamma(y,x)$ be the positive
fundamental solution of the sub-Laplacian
\[
\mathcal L\ =\ \sum_{j=1}^m X_j^*X_j ,
\]
and consider its level sets
\[
\Om(x,r)\ =\ \left\{y\in \Rn \mid \Gamma(x,y) > \frac{1}{r}\right\} .
\]

The following definition plays a key role in this paper.

\medskip

\begin{dfn}\label{D:lballs}
For every $x\in \Rn$, and $r>0$, the set
\[
B(x,r)\ =\ \left\{y\in \Rn \mid \Gamma(x,y) > \frac{1}{E(x,r)}\right\}
\]
will be called the $X$-\emph{ball}, centered at $x$ with radius $r$.
\end{dfn}

\medskip

We note explicitly that
\[
B(x,r)\ =\ \Om(x,E(x,r)),\quad \quad \text{and that} \quad \quad \Om(x,r)\ =\ B(x,F(x,r)).
\]

One of the main geometric properties of the $X$-balls, is that they
are equivalent to the Carnot-Carath\'eodory balls. To see this, we
recall the following  important result, established independently in
\cite{NSW}, \cite{SC}. Hereafter, the notation $Xu=(X_1u,...,X_mu)$
indicates the sub-gradient of a function $u$, whereas
$|Xu|=(\sum_{j=1}^m(X_ju)^2)^\frac{1}{2}$ will denote its length.

\medskip

\begin{thrm}\label{T:NSW}
 Given a bounded set $U\subset \Rn$, there exists $R_o$, depending on $U$ and on $X$, such that for $x\in U,\ 0<d(x,y)\leq R_o$, one has for $s\in \mathbb{N}\cup\{0\}$, and for some constant $C=C(U, X, s) >0$
\begin{align}\label{gradgamma}
& |X_{j_1}X_{j_2}...X_{j_s}\Gamma(x,y)|\ \leq\ C^{-1}\ \frac{d(x,y)^{2-s}}{|B_d(x,d(x,y))|},
\\
& \Gamma(x,y)\ \geq \ C\ \frac{d(x,y)^2}{|B_d(x,d(x,y))|} .
\notag
\end{align}
In the first inequality in \eqref{gradgamma}, one has $j_i\in \{1,...,m\}$ for $i=1,...,s$, and $X_{j_i}$ is allowed to act on either $x$ or $y$.
\end{thrm}

\medskip

In view of \eqref{dc}, \eqref{gradgamma}, it is now easy to recognize that, given a bounded set $U\subset \Rn$, there exists $a>1$, depending on $U$ and $X$, such that
\begin{equation}\label{equivi}
B_d(x,a^{-1}r)\ \subset\  B(x,r)\ \subset\ B_d(x,ar),
\end{equation}
for $x\in U, 0<r\leq R_o$. We observe that, as a consequence of \eqref{nsw2}, and of \eqref{gradgamma}, one has
\begin{equation}\label{F}
C\ d(x,y)\ \leq\  F\left(x,\frac{1}{\Gamma(x,y)}\right)\ \leq\  C^{-1}\ d(x,y),
\end{equation}
for all $x\in U, 0<d(x,y)\leq R_o$.

We observe that for a Carnot group $\bG$ of step $k$, if $\bg=V_1\oplus ...\oplus V_k$ is a stratification of the Lie algebra of $\bG$, then one has $\Lambda(x,r)=const\ r^Q$, for every $x\in \bG$ and every $r>0$, with $Q=\sum_{j=1}^k\;j\;dimV_j$, the homogeneous dimension of the group $\bG$.
In this case  $Q(x) \equiv Q$.

In the sequel the following properties of a Carnot-Carath\'eodory space will be useful.
\medskip

\begin{prop}\label{P:compact}
$(\Rn,d)$ is locally compact.  Furthermore, for any
bounded set $U\subset\Rn$ there exists $R_o=R_o(U)>0$ such that the closed
balls $\bar B(x_o,R)$, with $x_o\in U$ and $0<R<R_o$, are compact.
\end{prop}

\medskip

\begin{rmrk}
Compactness of balls of large radii may fail in general, see
\cite{GN1}.  However, there are important cases in which
Proposition~\ref{P:compact} holds globally, in the sense that one
can take $U$ to coincide with the whole ambient space and
$R_o=\infty$. One example is that of Carnot groups. Another
interesting case is that when the vector fields $X_j$ have
coefficients which are globally Lipschitz, see \cite{GN1},
\cite{GN2}. Henceforth, for any given bounded set $U\subset\Rn$ we
will always assume that the local parameter $R_o$ has been chosen so
to accommodate Proposition \ref{P:compact}.
\end{rmrk}

\section{\textbf{The Dirichlet problem}}\label{S:DP}

In what follows, given a system $X = \{X_1,...,X_m\}$ of $C^\infty$
vector fields in $\Rn$ satisfying \eqref{frc}, and an open set
$D\subset \Rn$, for $1\leq p\leq \infty$ we denote by $\mathcal
L^{1,p}(D)$ the Banach space $\{f\in L^p(D)\mid X_jf\in L^p(D),
j=1,...,m\}$ endowed with its natural norm
\[
||f||_{\mathcal L^{1,p}(D)}\ =\ ||f||_{L^p(D)}\ +\ \sum_{j=1}^m\ ||X_jf||_{L^p(D)}\ .
\]

The local space $\mathcal L^{1,p}_{loc}(D)$ has the usual meaning,
whereas for $1\leq p < \infty$ the space $\mathcal L^{1,p}_0(D)$ is
defined as the closure of $C^\infty_0(D)$ in the norm of $\mathcal
L^{1,p}(D)$. A function $u\in \mathcal L^{1,2}_{loc}(D)$ is called
\emph{harmonic} in $D$ if for any $\phi\in C^\infty_0(D)$ one has
\[
\int_D\ \sum_{j=1}^m X_j u X_j\phi\ dx\ =\ 0\ ,
\]
i.e., a harmonic function is a weak solution to the equation
$\mathcal L u = \sum_{j=1}^m X^*_jX_j u = 0$. By H\"ormander's
hypoellipticity theorem \cite{H}, if $u$ is harmonic in $D$, then
$u\in C^\infty(D)$. Given a bounded open set $D\subset \Rn$, and a
function $\phi\in \mathcal L^{1,2}(D)$, the Dirichlet problem
consists in finding  $u\in \mathcal L^{1,2}_{loc}(D)$ such that
\begin{equation}\label{DP2}
\begin{cases}
\mathcal Lu\ =\ 0 \quad\quad\quad\text{in}\quad D \ ,
\\
u\ -\ \phi\ \in \mathcal L^{1,2}_0(D)\ .
\end{cases}
\end{equation}

By adapting classical arguments, see for instance \cite{GT}, one can
show that there exists a unique solution $u\in \mathcal L^{1,2}(D)$
to \eqref{DP2}. If we assume, in addition, that $\phi\in C(D)$, in
general we cannot say that the function $u$ takes up the boundary
value $\phi$ with continuity. A Wiener type criterion for
sub-Laplacians was proved in \cite{NS}. Subsequently, using the
Wiener series in \cite{NS}, Citti obtained in \cite{Citti} an
estimate of the modulus of continuity at the boundary of the
solution of \eqref{DP2}. In \cite{D} an integral Wiener type
estimate at the boundary was established for a general class of
quasilinear equations having $p-$growth in the sub-gradient. Since
such estimate is particularly convenient for the applications, we
next state it for the special case $p=2$ of linear equations.

\begin{thrm}\label{T:Wiener}
Let $\phi\in \mathcal L^{1,2}(D)\cap C(\overline D)$. Consider the solution $u$ to \eqref{DP2}. There exist $C = C(X)>0$, and $R_o =R_o(D,X)>0$, such that given $x_o\in \partial D$, and $0<r<R<R_o/3$, one has
\begin{align*}
& osc\ \{u, D\cap B_d(x_o,r)\}\ \leq\ osc\ \{\phi, \partial D \cap \overline B_d(x_o,2R)\}
\\
& +\ osc\ (\phi, \partial D)\ \exp\ \left\{-\ C\
\int_r^R\left[\frac{cap_X\ (D^c\cap \overline B_d(x_o,t),
B_d(x_o,2t))}{cap_X\ (\overline B_d(x_o,t), B_d(x_o,2t))}\right]\
\frac{dt}{t}\right\} .
\end{align*}

\end{thrm}

In Theorem \ref{T:Wiener}, given a condenser $(K,\Om)$, we have
denoted by $cap_X(K,\Om)$ its Dirichlet capacity with respect to the
subelliptic energy $\mathcal E_X(u) = \int_\Om |Xu|^2 dx$ associated
with the system $X=\{X_1,...,X_m\}$. For the relevant properties of
such capacity we refer the reader to \cite{D}, \cite{CDG4}. A point
$x_o\in
\partial D$ is called regular if, for any $\phi\in \mathcal
L^{1,2}(D)\cap C(\overline D)$, one has
\begin{equation}\label{cont}
\lim_{x\to x_o}\ u(x)\ =\ \phi(x_o)\ .
\end{equation}

If every $x_o\in \partial D$ is regular, we say that $D$ is regular.
Similarly to the classical case, in the study of the Dirichlet
problem an important notion is that of generalized, or
Perron-Wiener-Brelot (PWB) solution to \eqref{DP2}. For operators of
H\"ormander type the construction of a PWB solution was carried in
the pioneering work of Bony \cite{B}, where the author also proved
that sub-Laplacians satisfy an elliptic type strong maximum
principle. We state next one of the main results in \cite{B} in a
form which is suitable for our purposes.

\begin{thrm}\label{T:MP}
Let $ D\subset \Rn$ be a connected, bounded open set, and $\phi\in
C(\partial D)$. There exists a unique harmonic function $H_{\phi}^{
D}$ which solves \eqref{DP} in the sense of Perron-Wiener-Brelot.
Moreover, $H_{\phi}^{ D}$ satisfies
\begin{equation}\label{pwb}
\underset{ D}{sup}\ |H_{\phi}^{ D}|\ \leq\ \underset{\partial{
D}}{sup}\ |\phi|\ .
\end{equation}
\end{thrm}

Theorem \ref{T:MP} allows to define the harmonic measure $d\omega^x$
for  $ D$ evaluated at $x\in  D$ as the unique probability measure
on $\partial{ D}$ such that for every $\phi\in C(\partial{ D})$
\[
H^{D}_{\phi}(x)\ =\ \int_{\partial{ D}}\ \phi(y)\ d\omega^x(y), \quad \quad x\in  D.
\]

A uniform Harnack inequality was established, independently, by
several authors, see \cite{X}, \cite{CGL}, \cite{L}: If $u$ is $\mathcal L-$harmonic in $D\subset \R^n$ and non-negative then there exists
$C,a>0$ such that for each ball $B(x,ar)\subset D$ one has
\begin{equation}\label{harnack-inv}
\sup_{B(x,r)} u \le C \inf_{B(x,r)} u .
\end{equation}
 Using such
Harnack principle one sees that for any $x, y\in D$, the measures
$d\omega^x$ and $d\omega^y$ are mutually absolutely continuous. For
the basic properties of the harmonic measure we refer the reader to
the paper \cite{CG}. Here, it is important to recall that, thanks to
the results in \cite{B}, \cite{CG}, the following result of Brelot
type holds.

\begin{thrm}\label{T:Brelot}
A function $\phi$ is resolutive if and only if $\phi\in
L^1(\partial{ D}, d\omega^x)$, for one (and therefore for all) $x\in
D$.
\end{thrm}

The following definition is particularly important for its
potential-theoretic implications. In the sequel, given a condenser
$(K,\Om)$, we denote by $cap(K,\Om)$ the sub-elliptic capacity of
$K$ with respect to $\Om$, see \cite{D}.

\begin{dfn}\label{D:thin}
An open set $D\subset \Rn$ is called \emph{thin} at $x_o\in \partial D$, if
\begin{equation}\label{positive}
\liminf_{r\to 0}\ \frac{cap_X(D^c \cap \overline B_d(x_o,r),B_d(x_o,
2r) )}{cap_X(\overline B_d(x_o,r),B_d(x_o,2r))}\
>\ 0 .
\end{equation}
\end{dfn}

\begin{thrm}\label{T:thin}
If a bounded open set $D\subset \Rn$ is thin at $x_o\in \partial D$, then $x_o$ is regular for the Dirichlet problem.
\end{thrm}

\begin{proof}[\textbf{Proof}]
 If $D$ is thin at $x_o\in \partial D$, then
\[
\int_0^R\left[\frac{cap_X(D^c\cap \overline B_d(x_o,t),
B_d(x_o,2t))}{cap_X(\overline B_d(x_o,t), B_d(x_o,2t))}\right]\
\frac{dt}{t}\ =\ \infty .
\]

Thanks to Theorem \ref{T:Wiener}, the divergence of the above integral implies for $0<r<R/3$
\[
osc\ \{u, D\cap B_d(x_o,r)\}\ \leq\ osc\ \{\phi, \partial D \cap \overline B_d(x_o,2R)\}\ .
\]

Letting $R\to 0$ we infer the regularity of $x_o$.

\end{proof}

A useful, and frequently used, sufficient condition for regularity is provided by the following definition.

\begin{dfn}\label{D:posden}
An open set $\Om\subset \Rn$ is said to have \emph{positive density} at $x_o\in \partial \Om$, if one has
\[
 \liminf_{r\to 0}\ \frac{|\Om \cap B_d(x_o,r)|}{|B_d(x_o,r)|}\ >\ 0 .
\]
\end{dfn}

\begin{prop}\label{P:posden}
If $D^c$ has positive density at $x_o$, then $D$ is thin at $x_o$.
\end{prop}

\begin{proof}[\textbf{Proof}]
We recall the Poincar\'e inequality
\[
\int_{\Om} |\phi|^2\ dx\ \leq\ C\ (diam(\Om))^2\ \int_\Om |X\phi|^2\ dx\ ,
\]
valid for any bounded open set $\Om\subset \Rn$, and any $\phi\in
C^1_o(\Om)$, where $diam(\Om)$ represents the diameter of $\Om$ with
respect to the distance $d(x,y)$, and $C = C(\Om,X)>0$. From the
latter, we obtain
\begin{equation}\label{1}
\frac{cap_X (D^c\cap \overline B_d(x_o,r), B_d(x_o,2r))}{cap_X
(\overline B_d(x_o,r), B_d(x_o,2r))}\ \geq\ \frac{C}{r^2}\
\frac{|D^c\cap \overline B_d(x_o,r)|}{cap_X (\overline B_d(x_o,r),
B_d(x_o,2r))} .
\end{equation}

Now the capacitary estimates in \cite{D}, \cite{CDG3} give
\[
C\ r^{Q-2}\ \leq\ cap_X (\overline B_d(x_o,r), B_d(x_o,2r))\ \leq\
C^{-1}\ r^{Q-2}\ ,
\]
for some constant $C = C(\Om,X)>0$. Using these estimates in \eqref{1} we find
\[
\frac{cap_X(D^c\cap \overline B_d(x_o,r), B_d(x_o,2r))}{cap_X
(\overline B_d(x_o,r), B_d(x_o,2r))}\ \geq\ C^*\ \frac{|D^c\cap
\overline B_d(x_o,r)|}{|B_d(x_o,r)|} ,
\]
where $C^* = C^*(\Om,X)>0$. The latter inequality proves that if $D^c$ has positive density at $x_o$, then $D$ is thin at the same point.

\end{proof}

A basic example of a class of regular domains for the Dirichlet
problem is provided by the (Euclidean) $C^{1,1}$ domains in a Carnot
group of step $r=2$. It was proved in \cite{CG} that such domains
possess a scale invariant region of non-tangential approach at every
boundary point, hence they satisfy the positive density condition in
Proposition \ref{P:posden}. Thus, in particular, every such domain
is regular for the Dirichlet problem for any fixed sub-Laplacian on
the group. Another important example is provided by the
non-tangentially accessible domains (NTA domains, henceforth)
studied in \cite{CG}. Such domains constitute a generalization of
those introduced by Jerison and Kenig in the Euclidean setting
\cite{JK}, see Section \ref{S:RHI}.

\begin{dfn}\label{D:FS}
Let $D\subset \Rn$ be a bounded open set. For $0<\alpha\leq 1$, the
class $\Gamma^{0,\alpha}_d(D)$ is defined as the collection of all
$f\in C(D)\cap L^\infty(D)$, such that
\[
\underset{x,y\in D, x\ne y}{sup}\ \frac{|f(x)-
f(y)|}{d(x,y)^{\alpha}} \ <\ \infty.
\]
We endow $\Gamma^{0,\alpha}_d(D)$ with the norm
\[
||f||_{\Gamma^{0,\alpha}_d(D)}\ =\ ||f||_{L^\infty(D)}\ +\
\underset{x,y\in D, x\ne y}{sup}\ \frac{|f(x)-
f(y)|}{d(x,y)^{\alpha}}\ .
\]
\end{dfn}

The meaning of the symbol $\Gamma^{0,\alpha}_{loc}(D)$ is the
obvious one, that is, $f\in \Gamma^{0,\alpha}_{loc}(D)$ if, for
every $\omega \subset\subset D$, one has $f\in
\Gamma_d^{0,\alpha}(\omega)$. If $F\subset \Rn$ denotes a bounded
closed set, by $f\in \Gamma_d^{0,\alpha}(F)$ we mean that $f$
coincides on the set $F$ with a function $g\in
\Gamma_d^{0,\alpha}(D)$, where $D$ is a bounded open set containing
$F$. The Lipschitz class $\Gamma^{0,1}_d(D)$ has a special interest,
due to its connection with the Sobolev space $\mathcal
L^{1,\infty}(D)$. In fact, we have the following theorem of
Rademacher-Stepanov type, established in \cite{GN1}, which will be
needed in the proof of Lemma \ref{L:Gamma}.

\begin{thrm}\label{T:meanvalue}
\textbf{(i)}\ Given a bounded open set $U\subset \Rn$, there exist
$R_o = R_o(U,X)>0$, and $C = C(U,X)>0$, such that if $f\in\mathcal
L^{1,\infty}(B_d(x_o,3R))$, with $x_o\in U$ and $0<R<R_o$, then $f$
can be modified on a set of $dx$-measure zero in $\bar B_d = \bar
B_d(x_o,R)$, so as to satisfy for every $x, y\in\bar B_d(x_o,R)$
\[
|f(x)\ -\ f(y)|\ \leq\  C\ d(x,y)\ \|f\|_{\mathcal L^{1,\infty}(B_d(x_o,3R))} .
\]
If, furthermore, $f\in C^{\infty}(B_d(x_o,3R))$, then in the right-hand side of the previous inequality one can replace the term $||f||_{\mathcal L^{1,\infty}(B_d(x_o,3R))}$ with $||Xf||_{L^{\infty}(B_d(x_o,3R))}$.

\noindent \textbf{(ii)}\ Vice-versa, let $D \subset \Rn$ be an open
set such that $\sup_{x,y\in D}\ d(x,y) < \infty$. If $f\in
\Gamma^{0,1}_d(D)$, then $f\in \mathcal L^{1,\infty}(D)$.
\end{thrm}

We note explicitly that part (i) of Theorem \ref{T:meanvalue}
asserts that every function $f\in\mathcal L^{1,\infty}(B_d(x_o,3R))$ has a representative which is Lipschitz continuous in $B_d(x_o,R)$ with respect to the metric $d$, i.e., continuing to denote with $f$ such representative, one has $f\in \Gamma^{0,1}(B_d(x_o,R))$. Part (ii) was also obtained independently in \cite{FSS}. The following result was established in \cite{D}.

\begin{thrm}\label{T:holder}
Let $D \subset \Rn$ be a bounded open set which is thin at every
$x_o\in \partial D$. If $\phi\in \Gamma^{0,\beta}(\overline D)$, for
some $\beta \in (0,1)$, then there exists $\alpha \in (0,1)$, with
$\alpha = \alpha(D,X,\beta)$, such that
\[
\underset{x,y\in \overline{D}, x\ne y}{sup}\ \frac{|H^{D}_{\phi}(x) - H^{D}_{\phi}(y)|}{d(x,y)^{\alpha}} \ <\ \infty\ .
\]
\end{thrm}

Given a bounded open set $D\subset \Rn$, consider the positive Green function $G(x,y) = G(y,x)$ for $\mathcal{L}$ and $D$, constructed in \cite{B}. For every fixed $x\in D$, one can represent $G(x,\cdot)$ as follows
\begin{equation}\label{represent}
G(x,\cdot)\ =\ \Gamma(x,\cdot)\ -\ h_x\ ,\quad\quad\text{where}\quad h_x\ =\ H^D_{\Gamma(x,\cdot)}\ .
\end{equation}

Since, by H\"ormander's hypo-ellipticity theorem, $\Gamma(x,\cdot)\in C^\infty(\Rn\setminus\{x\})$, we conclude that, if $D$ is thin at every $x_o\in \partial D$, then there exists $\alpha\in (0,1)$
such that, for every $\epsilon >0$, one has
\begin{equation}\label{alpha}
G(x,\cdot)\ \in\ \Gamma^{0,\alpha}_d(\bar  D \setminus B(x_o,\epsilon))\ .
\end{equation}

We close this section with recalling an important consequence of the
results  of Kohn and Nirenberg \cite{KN1} (see Theorem 4), and of
Derridj \cite{De1}, \cite{De2}, about smoothness in the Dirichlet
problem at non-characteristic points. We recall the following
definition.

\begin{dfn}\label{D:char}
Given a $C^1$ domain $D\subset \Rn$, a point $x_o\in \p D$ is called
characteristic for the system $X = \{X_1,...,X_m\}$ if for
$j=1,...,m$ one has \[ <X_j(x_o),\bN(x_o)>\ =\ 0\ , \] where
$\bN(x_o)$ indicates a normal vector to $\p D$ at $x_o$. We indicate
with $\Sigma = \Sigma_{D,X}$ the collection of all characteristic
points. The set $\Sigma$ is a closed subset of $\p D$.
\end{dfn}

\begin{thrm}\label{T:KN}
Let $D\subset \Rn$ be a $C^\infty$ domain which is regular for
\eqref{DP}. Consider the harmonic function $H^D_\phi$, with  $\phi
\in C^\infty(\partial{D})$. If $x_o\in \partial{D}$ is a
non-characteristic point for $\mathcal{L}$, then there exists an
open neighborhood $V$ of $x_o$ such that $H^D_\phi \in
C^\infty(\overline{D}\cap V)$.
\end{thrm}

\begin{rmrk}\label{R:J}
We stress that, as we indicated in the introduction, the conclusion
of Theorem \ref{T:KN} fails in general at characteristic points. In
fact, it fails so completely that even if the domain $D$ and the
boundary datum $\phi$ are real analytic, in general the solution of
the Dirichlet problem $H^D_\phi$ may be not better that H\"older
continuous up to the boundary, see Theorem \ref{T:holder}. An
example of such negative phenomenon in the Heisenberg group $\Hn$
was constructed by Jerison in \cite{J1}. The next section is
dedicated to it. For a related example concerning the heat equation
see \cite{KN2}.
\end{rmrk}

\section{\textbf{The example of D.
Jerison}}\label{S:Jerison}

Consider the Heisenberg group (discussed in the introduction) with
its left-invariant generators \eqref{vf} of its Lie algebra. Recall
that $\Hn$ is equipped with the non-isotropic dilations
\[
\delta_\lambda(z,t)\ =\ (\lambda z,\lambda^2 t)\ ,
\]
whose infinitesimal generator is given by the vector field
\[
\mathcal Z\ =\ \sum_{i=1}^n \bigg(x_i \frac{\p}{\p x_i} +  y_i
\frac{\p}{\p y_i}\bigg)\ +\ 2 \frac{\p}{\p t}\ .
\]

We say that a function $u:\Hn \to \R$ is \emph{homogeneous of
degree} $\alpha\in \R$ if for every $(z,t)\in \Hn$ and every
$\lambda>0$ one has
\[
u(\delta_\lambda(z,t))\ =\ \lambda^\alpha\ u(z,t)\ .
\]

One easily checks that if $u\in C^1(\Hn)$ then $u$ is homogeneous of
degree $\alpha$ if and only if \[ \mathcal Zu\ =\ \alpha\ u\ . \]

We also consider the vector field \begin{equation}\label{theta}
\Theta\ =\ \sum_{i=1}^n \bigg(x_i \frac{\p}{\p y_i} -  y_i
\frac{\p}{\p x_i}\bigg)\ , \end{equation} which is the infinitesimal
generator of the one-parameter group of transformations
$R_\theta:\Hn\to \Hn$, $\theta\in \R$, given by
\[
R_\theta(z,t)\ =\ (e^{i\theta}z,t),\ \ \ z = x + i y\in \mathbb C^n\
.
\]

Notice that when $n=1$, then in the $z$-plane $R_\theta$ is simply a
counterclockwise rotation of angle $\theta$, and in such case in the
standard polar coordinates $(r,\theta)$ in $\mathbb C$ we have
\[
\Theta\ =\ \frac{\p}{\p \theta}\ .
\]

In the sequel we will tacitly identify $z= x+iy \simeq(x,y)\in
\R^{2n}$, and so $|z| = \sqrt{|x|^2 + |y|^2}$. We note explicitly
that in the real coordinates $(x,y,t)$ the real part of the
Kohn-Spencer sub-Laplacian \eqref{kohn} on $\Hn$ is given by
\[
\mathcal L_o\ =\ \sum_{i=1}^{2n} X_i^2\ =\ \Delta_z +
\frac{|z|^2}{4} \frac{\p^2}{\p t^2}\ +\ \frac{\p}{\p t} \Theta\ .
\]

It is easy to see that if $u$ has \emph{cylindrical symmetry}, i.e.,
if
\[
u(z,t)\ =\ f(|z|,t)\ ,
\]
then
\[
\Theta u\ \equiv\ 0\ .
\]

Consider the gauge in $\Hn$
\[
N\ =\ N(z,t)\ =\ (|z|^4 + 16 t^2)^{1/4}\ .
\]

The following formula follows from an explicit calculation
\begin{equation}\label{hgN}
\psi\ \overset{def}{=}\ |\nh N|^2\ =\ \frac{|z|^2}{N^2}\ ,\ \ \
\Delta_H N\ =\ \frac{Q-1}{N}\ ,
\end{equation}
where
\[
Q = 2n + 2
\]
is the so-called \emph{homogeneous dimension} associated with the
non-isotropic dilations $\{\delta_\lambda\}_{\lambda>0}$. As a
consequence of \eqref{hgN}, if $u = f \circ N$ for some function
$f:[0,\infty)\to \R$, then one has the beautiful formula
\begin{equation}\label{slr}
\mathcal L_o u\ =\ \psi\ \bigg[f''(N)\ +\ \frac{Q-1}{N} f'(N)\bigg]\
.
\end{equation}

Since $f(t) = t^{2-Q}$ satisfies the ode in the right-hand side of
\eqref{slr} one can show that a fundamental solution of $- \mathcal
L_o$ with pole at the group identity $e = (0,0)\in \Hn$ is given by
\begin{equation}\label{ffs}
\Gamma(z,t)\ =\ \frac{C_Q}{N(z,t)^{Q-2}}\ ,\ \ (z,t)\not= e\ ,
\end{equation}
where $C_Q>0$ needs to be appropriately chosen.

The following example due to D. Jerison \cite{J1} shows that, even
when the domain and the boundary data are real analytic, in general
the solution to the subelliptic Dirichlet problem \eqref{DP} may not
be any better than $\Gamma^{0,\alpha}$ near a characteristic
boundary point. Consider the domain
\[
\Om_M\ =\ \{(z,t)\in \Hn\mid t> M |z|^2\}\ ,\quad\quad\quad M\in \R\
.
\]

Since $\Om_M$ is scale invariant with respect to
$\{\delta_\lambda\}_{\lambda>0}$ we might think of $\Om_M$ as the
analogue of a \emph{convex cone} ($M\geq 0$), or a \emph{concave
cone} ($M<0$). Introduce the variable
\[
\tau\ =\ \tau(z,t)\ =\ \frac{4t}{N^2}\ ,\ \ \ (z,t)\not= e\ .
\]

It is clear that $\tau$ is homogeneous of degree zero and therefore
\begin{equation*}\label{ztau}
\mathcal Z \tau\ =\ 0\ .
\end{equation*}

Moreover, with $\Theta$ as in \eqref{theta} , one easily checks that
\begin{equation*}\label{ttau}
\Theta \tau\ =\ 0\ .
\end{equation*}

It is important to observe the level sets $\{\tau = \gamma\}$ are
constituted by the $t$-axis when $\gamma = 1$, and by the
paraboloids
\[ t\ =\ \frac{\gamma}{4 \sqrt{1 - \gamma^2}} |z|^2\ ,
\]
if $|\gamma|<1$. Furthermore, the function $\tau$ takes the constant
value
\begin{equation*}\label{cv}
\tau\ =\ \frac{4M}{\sqrt{1 + 16 M^2}}\ ,
\end{equation*}
on $\p \Om_M$. We now consider a function of the form
\begin{equation}\label{vJ}
v\ =\ v(z,t)\ =\ N^\alpha \ u(\tau)\ ,
\end{equation}
where the number $\alpha>0$ will be appropriately chosen later on.
One has the following result whose verification we leave to the
reader.

\begin{prop}\label{P:jerison}
For any $\alpha>0$ one has
\begin{align*}
\mathcal L_o v\ & =\ 4 \psi N^{\alpha - 2} \bigg\{(1 - \tau^2)
u''(\tau) - \frac{Q}{2} \tau u'(\tau) + \frac{\alpha(\alpha + Q -
2)}{4} u(\tau)\bigg\}
\\
& =\ 4 \psi N^{\alpha - 2} \bigg\{(1 - \tau^2) u''(\tau) - (n+1)
\tau u'(\tau) + \frac{\alpha(\alpha + 2n)}{4} u(\tau)\bigg\} .
\end{align*}
\end{prop}

Using Proposition \ref{P:jerison} we can now construct a positive
harmonic function in $\Om_M$ which vanishes on the boundary (this
function is a Green function with pole at an interior point).

\begin{prop}\label{P:example}
For any $\alpha\in (0,1]$ there exists a number $M = M(\alpha)<0$
such that the nonconvex cone $\Om_M$ admits a positive solution of
$\mathcal L_o v = 0$ of the form \eqref{vJ} which vanishes on $\p
\Om_M$.
\end{prop}

\begin{proof}[\textbf{Proof}]
From Proposition \ref{P:jerison} we see that if $v$ of the form
\eqref{vJ} has to solve the equation $\mathcal L_o v = 0$, then the
function $u$ must be a solution of the Jacobi equation
\begin{equation}\label{je}
(1 - \tau^2) u''(\tau) - (n+1) \tau u'(\tau) + \frac{\alpha(\alpha +
2n)}{4} u(\tau)\ =\ 0\ .
\end{equation}

As we have observed the level $\{\tau = 1\}$  is degenerate and
corresponds to the $t$-axis $\{z = 0\}$. One solution of \eqref{je}
which is smooth as $\tau \to 1$ (remember, the $t$-axis is inside
$\Om_M$ and thus we want our function $v$ to be smooth around the
$t$-axis since by hypoellipticity $v$ has to be in
$C^\infty(\Om_M)$) is the hypergeometric function
\[
g_\alpha(\tau)\ =\ F\left(- \frac{\alpha}{2},n +
\frac{\alpha}{2};\frac{n+1}{2};\frac{1-\tau}{2}\right)\ .
\]

When $0<\alpha<2$ one can varify that
\[
g_\alpha(1) = 1\ ,\ \ \ \text{and that}\ \ g_\alpha(\tau) \to -
\infty\ \text{as}\ \tau \to - 1^+\ .
\]

Therefore, $g_\alpha$ has a zero $\tau_\alpha$. One can check (see
Erdelyi, Magnus, Oberhettinger and Tricomi, vol.1, p.110 (14)), that
as $\alpha \to 0^+$, then $\tau_\alpha \to -1^+$. We infer that for
$\alpha>0$ sufficiently close to $0$ there exists $-1<\tau_\alpha<0$
such that
\[
g_\alpha(\tau_\alpha)\ =\ 0\ .
\]

If we choose
\[
M\ =\ M(\alpha)\ =\ \frac{\tau_\alpha}{\sqrt{1 - \tau_\alpha^2}}\ <\
0\ ,
\]
then it is clear that on $\p \Om_M$ we have $\tau \equiv
\tau_\alpha$, and therefore the function $v$ of the form \eqref{v},
with $u(\tau) = g_\alpha(\tau)$, has the property of being harmonic
and nonnegative in $\Om$, and furthermore on $\p \Om_M$ we have that
$v = N^\alpha g_\alpha(\tau_\alpha) \equiv 0$. This completes the
proof.

\end{proof}

Since $\alpha$ belongs the interval $(0,1)$, then it is clear that
$v = N^\alpha(z,t) g_\alpha(\tau)$ belongs at most to the
Folland-Stein H\"older class $\Gamma^{0,\alpha}(\overline \Om_M)$,
but is not any better than metrically H\"older in any neighborhood
of $e = (0,0)$. What produces this negative phenomenon is the fact
that the point $e\in \p \Om_M$ is characteristic for $\Om_M$.

\section{\textbf{Subelliptic interior Schauder estimates}}\label{S:SE}

In this section we establish some basic interior Schauder type
estimates that, besides from playing an important role in the
sequel, also have an obvious independent interest. Such estimates
are tailored on the intrinsic geometry of the system $X =
\{X_1,...,X_m\}$, and are obtained by means of a family of
sub-elliptic mollifiers which were introduced in \cite{CDG1}, see
also \cite{CDG2}. For convenience, we state the relevant results in
terms of the $X$-balls $B(x,r)$ introduced in Definition
\ref{D:lballs}, but we stress that, thanks to \eqref{equivi}, we
could have as well employed the metric balls $B_d(x,r)$. Since in
this paper our focus is on $\mathcal L$-harmonic functions, we do
not explicitly treat the non-homogeneous equation $\mathcal L u = f$
with a non-zero right-hand side. Estimates for solutions of the
latter equation can, however, be obtained by relatively simple
modifications of the arguments in the homogeneous case.

The following is the main result in this section.

\begin{thrm}\label{T:Harmonic}
Let $D\subset \Rn$ be a bounded open set and suppose that $u$ is
harmonic in $D$. There exists $R_o>0$, depending on $D$ and $X$,
such that for every $x\in D$ and $0<r\leq R_o$ for which
$\overline{B}(x,r)\subset D$, one has for any $s\in \mathbb N$
\[
|X_{j_1}X_{j_2}...X_{j_s}u(x)|\ \leq\ \frac{C}{r^s}\ \underset{\overline{B}(x,r)}{max}\ |u|,
\]
for some constant $C=C(D,X,s)>0$. In the above estimate, for every $i = 1,...,s,$ the index $j_i$ runs in the set $\{1,...,m\}$ .
\end{thrm}

\begin{rmrk}
We emphasize that Theorem \ref{T:Harmonic} cannot be established
similarly to its classical ancestor for harmonic functions, where
one uses the mean-value theorem coupled with the trivial observation
that any derivative of a harmonic function is harmonic. In the
present non-commutative setting, derivatives of harmonic functions
are no longer harmonic!
\end{rmrk}

A useful consequence of Theorem \ref{T:Harmonic} is the following.

\begin{cor}\label{C:Harnack}
Let $D\subset \Rn$ be a bounded, open set and suppose that $u$ is  a
non-negative harmonic function in $D$. There exists $R_o>0$,
depending on $D$ and $X$, such that for any $x\in D$ and $0<r\leq
R_o$ for which $\overline{B}(x,2r)\subset D$, one has for any given
$s\in \mathbb N$
\[
|X_{j_1}X_{j_2}...X_{j_s}u(x)|\ \leq\ \frac{C}{r^s}\ u(x),
\]
for some $C=C(D,X,s)>0$.
\end{cor}

\begin{proof}[\textbf{Proof}]
Since $u\geq 0$, we immediately obtain the result from Theorem \ref{T:Harmonic} and from the Harnack inequality \eqref{harnack-inv}.
\end{proof}

To prove Theorem \ref{T:Harmonic}, we use the family of sub-elliptic
mollifiers introduced in \cite{CDG1}, see also \cite{CDG2}. Choose a
nonnegative function $f\in C_o^{\infty}(\R)$, with $supp\ f\subset
[1,2]$, and such that $\int_\R f(s)ds=1$, and let
$f_R(s)=R^{-1}f(R^{-1}s)$. We define the kernel

\begin{equation*}
K_R(x,y)\ =\ f_R\left(\frac{1}{\Gamma(x,y)}\right)\ \frac{|X_y\Gamma(x,y)|^2}{\Gamma(x,y)^2} .
\end{equation*}

Given a function $u\in L_{loc}^1(\Rn)$, following \cite{CDG1} we
define the \emph{subelliptic mollifier} of $u$  by
\begin{equation}\label{JR}
J_R\ u(x)\ =\ \int_{\Rn}\ u(y)\ K_R(x,y)\ dy ,    \quad\quad  \quad \quad R>0.
\end{equation}

We note that for any fixed $x\in \Rn$,
\begin{equation}\label{supp}
supp\ K_R(x,\cdot)\ \subset\  \Om(x,2R)\ \setminus\ \Om(x,R) .
\end{equation}

One of the important features of $J_R\ u$ is expressed by the
following theorem.

\begin{thrm}\label{T:MeanValue}
Let $D\subset \Rn$ be open and suppose that $u$ is harmonic in $D$.
There exists $R_o>0$, depending on $D$ and $X$, such that for any
$x\in D$, and every $0<R\leq R_o$ for which
$\overline{\Om}(x,2R)\subset D$, one has
\[
u(x)\ =\ J_R\ u(x).
\]
\end{thrm}

\begin{proof}[\textbf{Proof}]
Let $u$ and $\Om(x,R)$ be as in the statement of the theorem. We obtain for $\psi \in C^{\infty}(D)$ and $0<t\leq R$, see \cite{CGL},
\begin{equation}\label{psi}
\psi(x)\ =\ \int_{\partial \Om(x,t)}\ \psi(y)\ \frac{|X_y\Gamma(x,y)|^2}{|D\Gamma(x,y)|}\ dH_{n-1}(y)\ +\ \int_{\Om(x,t)}\ \mathcal{L}\psi(y)\ \big[\Gamma(x,y)-\frac{1}{t}\big]\ dy.
\end{equation}

Taking $\psi=u$ in \eqref{psi}, we find
\begin{equation}\label{meanvalue}
u(x)\ =\ \int_{\partial \Om(x,t)}\ u(y)\ \frac{|X_y\Gamma(x,y)|^2}{|D\Gamma(x,y)|}\ dH_{n-1}(y).
\end{equation}

We are now going to use \eqref{meanvalue} to complete the proof. The idea is to start from the definition of $J_R\ u(x)$, and then use Federer co-area formula \cite{Fe}. One finds
\[
J_R\ u(x)\ =\ \int_0^{\infty}\ f_R(t)\ \left[\int_{\partial \Om(x,t)}\ u(y)\ \frac{|X_y\Gamma(x,y)|^2}{|D\Gamma(x,y)|}\ dH_{n-1}(y)\right]\ dt .
\]

The previous equality, \eqref{meanvalue}, and the fact that  $\int_\R f_R(s)ds=1$, imply the conclusion.

\end{proof}

The essence of our main a priori estimate is contained in the
following theorem.

\begin{thrm}\label{T:AprioriE}
Fix a bounded set $U\subset \Rn$. There exists a  constant $R_o>0$,
depending only on $U$ and on the system $X$, such that for any $u\in
L_{loc}^1(\Rn)$, $x\in U, 0<R\leq R_o$, and $s\in \mathbb N$ one has
for some $C=C(U,X,s)>0$,
\[
|X_{j_1}X_{j_2}...X_{j_s}\ J_R\ u(x)|\ \leq\  \frac{C}{R}\ \frac{1}{F(x,R)^{2+s}}\ \int_{\Om(x,R)}\ |u(y)|\ dy .
\]
\end{thrm}

\begin{proof}[\textbf{Proof}]

We first consider the case $s=1$. From \eqref{gradgamma}, and from the support property \eqref{supp} of $K_R(x,\cdot)$, we can differentiate under the integral sign in \eqref{JR}, to obtain
\[
|X\ J_R\ u(x)|\ \leq\ \int_{B(x,2R)}\ |u(y)|\ |X_xK_R(x,y)|\ dy.
\]

By the definition of $K_R(x,y)$ it is easy to recognize that the components of its sub-gradient $X_xK_R(x,y)$ are estimated as follows
\begin{align}
|X_j\ K_R(x,y)|\ &\leq\   C\ R^{-2}\ |X\Gamma(x,y)|^3\ \Gamma(x,y)^{-4}\
\notag\\
&+\  C\ R^{-1}\ \Gamma(x,y)^{-2}\ \sum_{k=1}^m|X_jX_k\Gamma(x,y)|\ |X_k\Gamma(x,y)|\
\notag\\
&+\ C\ R^{-1}\ |X\Gamma(x,y)|^3\ \Gamma(x,y)^{-3}\
\notag\\
&=\  I_R^1(x,y)\ +\ I_R^2(x,y)\ +\ I_R^3(x,y) .
\notag
\end{align}

To control the three terms in the right-hand side of the above inequality, we use the size estimates \eqref{gradgamma}, along with the observation that, due to the fact that on the support of $K_R(x,\cdot)$ one has
\[
\frac{1}{2R}\ <\ \Gamma(x,y)\ \leq\ \frac{1}{R} ,
\]
then Theorem \ref{T:NSW}, and \eqref{F}, give for all $x\in U, 0<R\leq R_o$, and $y\in \Om(x,2R)\setminus \Om(x,R)$
\begin{equation}\label{dF}
C\ \leq\ \frac{d(x,y)}{F(x,R)}\ \leq\ C^{-1} .
\end{equation}

Using \eqref{gradgamma}, \eqref{dF},  one obtains that for $i =1, 2, 3$
\[
\underset{y\in \Om(x,2R)\setminus \Om(x,R)}{sup}|I_R^i(x,y)|\leq \frac{C}{RF(x,R)^3}
\]
for any $x\in U$, provided that $0<R\leq R_o$. This completes the proof in the case $s=1$. The case $s\geq 2$ is handled recursively by similar considerations based on Theorem \ref{T:NSW}, and we omit details. It may be helpful for the interested reader to note that Theorem \ref{T:NSW} implies
\[
|X_{j_1}X_{j_2}...X_{j_s}\ \Gamma(x,y)|\ \leq\ C\ d(x,y)^{-s}\ \Gamma(x,y),
\]
so that by \eqref{dF} one obtains
\begin{equation}\label{gammaF}
\underset{y\in \Om(x,2R)\setminus \Om(x,R)}{sup}\ |X_{j_1}X_{j_2}...X_{j_s}\ \Gamma(x,y)|\ \leq \ \frac{C}{RF(x,R)^s} .
\end{equation}

\end{proof}

We are finally in a position to prove Theorem \ref{T:Harmonic}.

\begin{proof}[\textbf{Proof of Theorem \ref{T:Harmonic}}]
We observe explicitly that the assumption states that with $R=E(x,r)/2$, then $\overline{\Om}(x,2R)=\overline{B}(x,r)\subset D$. By Theorem \ref{T:MeanValue}, and by \eqref{gammaF}, we find
\begin{align*}
& |X_{j_1}X_{j_2}...X_{j_s}\ u(x)|\ =\ |X_{j_1}X_{j_2}...X_{j_s}(J_R\ u)(x)|
\notag\\
& \leq\ \frac{C}{R F(x,R)^{2+s}}\  \int_{\Om(x,R)}\ |u(y)|\ dy\ \leq\  C\ \frac{|\Om(x,R)|}{R F(x,R)^{2+s}}\ \  \underset{\overline{\Om}(x,R)}{max}\ |u| .
\notag
\end{align*}

To complete the proof we only need to observe that $\Om(x,R) = B(x,r)$, and that, thanks to Theorem \ref{T:NSW}, \eqref{F}, one has
\[
\frac{C}{r^s}\ \leq\ \frac{|B(x,r)|}{R F(x,R)^{2+s}}\ \leq\  \frac{C^{-1}}{r^s} .
\]

\end{proof}

\begin{rmrk}\label{R:groups}
We observe explicitly that when $\bG$ is a Carnot group with
$X_1,...,X_m$ being a fixed basis of the horizontal layer of its Lie
algebra, then the constant $C$ in Theorem \ref{T:Harmonic} and
Corollary \ref{C:Harnack} can be taken independent of the open set
$D$.
\end{rmrk}

\section{\textbf{Lipschitz boundary estimates for the Green function}}\label{S:bdry}

In this section we establish some basic estimates at the boundary
for the Green function associated to a sub-Laplacian, when the
relevant domain possesses an appropriate analogue of the outer
tangent sphere condition introduced by Poincar\'e in his famous
paper \cite{P}. Analyzing the domain $\Om_M$ in Remark \ref{R:J} one
recognizes that Jerison's negative example fails to possess a
tangent outer gauge sphere at its characteristic point. We thus
conjectured that by imposing such condition one should be able to
establish the boundedness near the boundary of the horizontal
gradient of the Green function (see for instance \cite{G} for the
classical case of elliptic or parabolic operators). This intuition
has proved correct. In their paper \cite{LU} Lanconelli and Uguzzoni
have proved the boundedness of the Poisson kernel for a domain
satisfying the outer sphere condition in the Heisenberg group,
whereas in \cite{CGN2} a similar result was successfully combined
with those in \cite{CG} to obtain a complete solution of the
Dirichlet problem for a large class of domains in groups of
Heisenberg type.

The objective of this section is to generalize the cited results in
\cite{LU} and \cite{CGN2} to the Poisson kernel associated with an
operator of H\"ormander type. Namely,  {\it if $D\subset \Rn$ is a bounded domain
satisfying an intrinsic uniform outer sphere condition with respect
to a system $X = \{X_1,...,X_m\}$ satisfying \eqref{frc}, and having
Green function $G(x,y) = G_D(x,y)$, if we fix the singularity at an
interior point $x_1\in D$, then the function $x\to |XG(x_1,x)|$,
which is well defined for $x\in D\setminus \{x_1\}$, belongs to
$L^\infty$ in a neighborhood of $\partial D$}. The exact statements
are contained in Corollaries \ref{C:Lipschitz} and
\ref{C:LipschitzSigma}.

 We emphasize that, in
view of Theorem \ref{T:KN}, the main novelty of this result lies in
that we do allow the boundary point to be characteristic. As it will
be clear from the analysis below, the passage from the group setting
to the case of general sub-Laplacians involves overcoming various
non-trivial obstacles.

Our first task is to obtain a growth estimate at the boundary for
harmonic functions which vanish on a distinguished portion of the
latter. We show that any such function grows at most linearly with
respect to the Carnot-Carath\'eodory distance associated to the
system $X$. The proof of this result ultimately relies on delicate
estimates of a suitable barrier whose construction  is inspired to
that given by Poincar\'e  \cite{P}, see also \cite{G}. We begin with
a lemma which plays a crucial role in the sequel. The function
$\Gamma(x,y)=\Gamma(y,x)$ denotes the positive fundamental solution
of the sub-Laplacian associated with the system $X$, see Section
\ref{S:Prelim}.

\begin{lemma}\label{L:Gamma}
 For any bounded set $U\subset\Rn$, there
exist $R_o, C>0$, depending on $U$ and $X$,  such that for every
$x_o\in U$, and $x, y\in \Rn\setminus B_d(x_o,r)$, one has
\[
 |\Gamma(x_o,x)-\Gamma(x_o,y)|\ \leq\ C\ \frac{r}{|B_d(x_o,r)|} \ d(x,y).
\]
\end{lemma}

\begin{proof}[\textbf{Proof}]
We distinguish two cases: (i) $d(x,y)> \theta r$; (ii) $d(x,y)\leq
\theta r$. Here, $\theta\in (0,1)$ is to be suitably chosen. Case
(i) is easy. Using \eqref{gradgamma} we find
\begin{align}
& |\Gamma(x_o,x)\ -\ \Gamma(x_o,y)|\  \leq\ \Gamma(x_o,x)\ +\ \Gamma(x_o,y)
\notag \\
& \leq\ C\ \left[\frac{d(x_o,x)^2}{|B_d(x_o,d(x_o,x))|}\  +\
\frac{d(x_o,y)^2}{|B_d(x_o,d(x_o,y))|}\right]
\notag \\
& \leq\ C\ \left\{\frac{1}{E(x_o,d(x_o,x))} +
\frac{1}{E(x_o,d(x_o,y))}\right\}\ \leq\ C\ \frac{1}{E(x_o,r)}\ <\
\frac{C}{\theta}\ \frac{r}{|B_d(x_o,r)|} \ d(x,y)\ . \notag
\end{align}

We next consider case (ii), and let $\rho = d(x,y)\leq \theta r$.
Let $\gamma$ be a sub-unitary curve joining $x$ to $y$ with length
$l_s(\gamma)\leq\rho + \rho/16$. The existence of such a curve is
guaranteed by the definition of $d(x,y)$. Consider the function
$g(P)\overset{def}{=} d(x,P)-d(y,P)$. By the continuity of
$g:\{\gamma\}\to\mathbb{R}$, and by the intermediate value theorem,
we can find $P\in\{\gamma\}$ such that $d(x,P)=d(y,P)$. For such
point $P$, we must have
\begin{equation}\label{P}
d(x,P)\ =\ d(y,P)\ \leq\ \frac{3}{4}\ \rho .
\end{equation}

If \eqref{P} were not true, we would in fact have
\[
\frac{3}{4}\rho\ +\ \frac{3}{4}\rho\ <\ d(x,P)\ +\ d(y,P)\ \leq\ l_s(\gamma)\ \leq\ \rho\ +\ \frac{\rho}{16},
\]
which is a contradiction. From \eqref{P} we conclude that $x,y\in B_d(P,3\rho/4)$.
Moreover,
\[
d(P,x_o)\ \geq\ d(x,x_o)\ -\ d(x,P)\ \geq\  r\ -\ \frac{3}{4}\rho\ \geq\ \left(1 - \frac{3}{4}\theta\right)\ r .
\]

We claim that
\begin{equation}\label{case2}
B_d(P,\frac{9}{4}\rho)\ \subset\ \Rn\ \setminus\ B_d(x_o,r/2),
\end{equation}
 provided that we take $\theta = \frac{1}{6}$. In fact, let $z\in B_d(P,\frac{9}{4} \rho)$, then
\[
d(z,x_o)\ \geq\ d(P,x_o) - d(z,P)\ \geq\ \left(1 - \frac{3}{4}\theta \right)\ r\ -\ \frac{9}{4}\ \theta\ r\ =\ (1 - \frac{3}{4}\theta - \frac{9}{4}\theta)\ r\ =\ \frac{r}{2} .
\]

This proves \eqref{case2}. The above considerations allow to apply
Theorem \ref{T:meanvalue}, which, keeping in mind that
$\Gamma(x_o,\cdot)\in C^\infty(B_d(P,\frac{9}{4}\rho))$, presently
gives
\begin{equation}\label{43}
|\Gamma(x_o,x)\ -\ \Gamma(x_o,y)|\ \leq\  C\ \rho\ \underset{\xi\in
B_d(P,\frac{9}{4}\rho)}{sup}|X\Gamma(x_o,\xi)|\ .
\end{equation}

Using \eqref{gradgamma} we obtain for $\xi\in B_d(P,\frac{9}{4}\rho)$
\[
|X\Gamma(x_o,\xi)|\ \leq\ C\ \frac{1}{d(x_o,\xi)\
E(x_o,d(x_o,\xi))}\ ,
\]
where $t\to E(x_o,t)$ is the function introduced in \eqref{E}.
Since by \eqref{case2} we have $d(x_o,\xi)\geq r/2$, the latter estimate, combined with the increasingness of $E(x_o,\cdot)$, leads to the conclusion
\[
\underset{\xi\in B_d(P,\frac{9}{4}\rho)}{sup}|X\Gamma(x_o,\xi)|\ \leq\  C\ \frac{1}{rE(x_o,r)}.
\]

Inserting this inequality in \eqref{43}, and observing that $\frac{1}{rE(x_o,r)} \leq  C  \frac{r}{|B_d(x_o,r)|}$, we find
\[
|\Gamma(x_o,x)\ -\ \Gamma(x_o,y)|\ \leq\ C\ \frac{r}{|B_d(x_o,r)|}\ d(x,y) .
\]

This completes the proof of the lemma.

\end{proof}

The following definition plays a crucial role in the subsequent
development.

\begin{dfn}\label{D:OB}
A domain $ D\subset \Rn$ is said to possess an \emph{outer $X$-ball
tangent} at $x_o\in \partial{ D}$ if for some $r>0$ there exists a
$X$-ball $B(x_1,r)$ such that:
\begin{equation}\label{ob}
x_o\ \in\ \partial{B}(x_1,r), \quad\quad \quad  B(x_1,r)\ \cap\ D\ =\ \varnothing.
\end{equation}
We  say that  $D$ possesses the \emph{uniform outer $X$-ball} if one
can find $R_o>0$  such that   for every $x_o\in \partial{ D}$, and
any $0<r<R_o$, there exists a $X$-ball $B(x_1,r)$ for which
\eqref{ob} holds.
\end{dfn}

Some comments are in order. First, it should be clear from
\eqref{equivi} that the existence of an outer $X$-ball tangent at
$x_o\in
\partial D$ implies that $D$ is thin at $x_o$ (the reverse
implication is not necessarily true). Therefore, thanks to Theorem
\ref{T:thin}, $x_o$ is regular for the Dirichlet problem. Secondly,
when $X =
\{\frac{\partial}{\partial{x_1}},...,\frac{\partial}{\partial{x_n}}\}$,
then the distance $d(x,y)$ is just the ordinary Euclidean distance
$|x-y|$. In such case, Definition \ref{D:OB} coincides with the
notion introduced by Poincar\'e in his classical paper \cite{P}. In
this setting a $X$-ball is just a standard Euclidean ball, then
every $C^{1,1}$ domain and every convex domain possess the uniform
outer $X$-ball condition. When we abandon the Euclidean setting, the
construction of examples is technically much more involved and we
discuss them in the last section of this paper.

We are now ready to state the first key boundary estimate.

\begin{thrm}\label{T:Growth}
Let $ D\subset \Rn$ be a connected open set, and suppose that for
some $r>0$, $D$ has an outer $X$-ball  $B(x_1,r)$ tangent at $x_o\in
\partial{ D}$. There exists $C>0$, depending only on $D$ and on
$X$, such that if $\phi\in C(\partial{ D})$,  $\phi\equiv 0 \quad
in\quad B(x_1,2r)\cap \partial{ D}$, then we have for every $x\in
D$
\[
|H_{\phi}^{ D}(x)|\ \leq\ C\ \frac{d(x,x_o)}{r}\ \underset{\partial{ D}}{max}\ |\phi|.
\]
\end{thrm}

\begin{proof}[\textbf{Proof}]
Without loss of generality we assume $\underset{\partial{D}}{max}|\phi|=1$. Following the idea in \cite{P} we introduce the function
\begin{equation}\label{deff}
f(x)\ =\ \frac{E(x_1,r)^{-1}\ -\ \Gamma(x_1,x)}{E(x_1,r)^{-1}\ -\ E(x_1,2r)^{-1}} ,\quad\quad\quad x\in D ,
\end{equation}
where $x\to \Gamma(x_1,x)$ denotes the positive fundamental solution of $\mathcal L$, with singularity at $x_1$, and $t \to E(x_1,t)$ is defined as in \eqref{E}.
Clearly, $f$ is $\mathcal{L}$-harmonic in $\Rn\setminus \{x_1\}$. Since $\Gamma(x_1,\cdot)\leq E(x_1,r)^{-1}$ outside $B(x_1,r)$, we see that $f\geq 0$ in $\Rn\setminus B(x_1,r)$, hence in particular in $D$. Moreover, $f\equiv 1$ on $\partial{B}(x_1,2r)\cap D$, whereas $f\geq 1$ in $(\Rn\setminus B(x_1,2r))\cap \overline{D}$ . By Theorem \ref{T:MP} we infer
\[
|H_{\phi}^{ D}(x)|\ \leq\ f(x) \quad\quad\quad\text{ for every}\quad x\ \in\ D .
\]

The proof will be completed if we show that
\begin{equation}\label{f}
f(x)\ \leq\ C\ \frac{d(x,x_o)}{r} ,\quad\quad\quad \text{for every}\quad  x\ \in\ D .
\end{equation}

Consider the function $h(t)=E(x_1,t)^{-1}$. We have for $0<s<t<R_o$,
\[
h(s)\ -\ h(t)\ =\ (t-s)\ \frac{E'(x_1,\tau)}{E(x_1,\tau)^2} ,
\]
for some $s<\tau <t$ . Using the increasingness of the function $r\to rE(x_1,r)$, which follows from that of $E(x_1,\cdot)$, and the crucial estimate
\[
C\ \leq\ \frac{rE'(x_1,r)}{E(x_1,r)}\ \leq\ C^{-1},
\]
which is readily obtained from the definition of $\Lambda(x_1,r)$ in \eqref{pol},
we find
\begin{equation}\label{h}
C\ \frac{t-s}{t E(x_1,t)}\ \leq\  h(s)\ -\ h(t)\ \leq\ C^{-1}\ \frac{t-s}{s E(x_1,s)}.
\end{equation}

Keeping in mind the definition \eqref{deff} of $f$, from \eqref{h}, and from the fact that $E(x_1,\cdot)$ is doubling, we obtain
\[
f(x)\ \leq\ C\ E(x_1,r)\ \{\Gamma(x_1,x_o)\ -\ \Gamma(x_1,x)\},
\]
where we have used the hypothesis that $x_o\in \partial{B}(x_1,r)$. The proof of \eqref{f} will be achieved if we show that for $x\in \Rn\setminus B(x_1,r)$
\[
\Gamma(x_1,x_o)\ -\ \Gamma(x_1,x)\ \leq\ C\ d(x,x_o)\ \frac{1}{r E(x_1,r)}.
\]

In view of \eqref{equivi}, the latter inequality follows immediately
from Lemma \ref{L:Gamma}. This completes the proof.

\end{proof}

Let $ D\subset \Rn$ be a domain. Consider the positive Green function $G(x,y)$ associated to $\mathcal {L}$ and $ D$.  From Theorem \ref{T:MP} and from the estimates \eqref{gradgamma} one easily sees that there exists a positive constant $C_{ D}$ such that for every $x, y\in  D$
\begin{equation}\label{green}
0\ \leq\ G(x,y)\ \le\ C_{ D}\ \frac{d(x,y)^2}{|B_d(x,d(x,y))|}\ ,
\end{equation}
for each $x,y\in  D$. Our next task is to obtain more refined estimates for $G$.

\begin{thrm}\label{T:Green}
Suppose that $D\subset \Rn$ satisfy the uniform outer $X$-ball
condition. There exists a constant $C=C(X, D)>0$ such that
\[
G(x,y) \le C \frac{d(x,y)}{|B_d(x,d(x,y))|}\ d(y,\partial{ D})
\]
for each $x,y\in  D$, with $x\neq y$.
\end{thrm}

\begin{proof}[\textbf{Proof}]
Consider $a>1$ as in \eqref{equivi}, and let $R_0$ be the constant
in Definition \ref{D:OB} of uniform outer $X$-ball condition. The
estimate that we want to prove  is immediate if one of the points is
away from the boundary. In fact, if
 either $d(y,\partial  D)\geq \frac{d(x,y)}{a^2(3+a)}$, or $d(y,\partial  D)\geq R_o$,
then the conclusion follows from \eqref{green}. We may thus assume
that
\begin{equation}\label{a}
a\  d(y,\partial{ D})< \frac{d(x,y)}{a(a+3)}, \quad \mbox{and} \quad d(y,\partial{ D})< R_o.
\end{equation}

We now choose
\[
r\ =\ min\left(\frac{d(x,y)}{2a(a+3)}, \frac{aR_o}{2}\right).
\]

One easily verifies from \eqref{a} that $a d(y,\partial{ D}) < 2r$.
Let $x_o$ be the point  in $\partial{ D}$ such that  $d(y,\partial
D) = d(y,x_o)$ and consider the outer $X$-ball $B(x_1,r/a)$ tangent
to the boundary of $D$ in $x_o$. We claim that
\[
y \in  D \cap B(x_1,(a+3)r).
\]

To see this observe that by \eqref{equivi}   $x_o\in
\overline{B}(x_1,\frac{r}{a}) \subset \overline{B}_d(x_1,r)$, and
therefore
\[
d(y,x_1)\  \le\  d(y,x_o)+d(x_o,x_1)\ =\
d(y,\partial  D)+d(x_o, x_1)\ \le\  \frac{a+2}{a}r\  <\  \frac{a+3}{a}r.
\]

This shows $y\in B_d(x_1,a^{-1}(a+3)r)$. Another application of
\eqref{equivi} implies the claim. Next,  the triangle inequality
gives
\[
d(x,x_1)\ \ge\ d(x,y)-d(x_1,y)\ \ge\ d(x,y) - \frac{a+3}{a}r\ \geq\ d(x,y)(1 - \frac{1}{2a^2}),
\]
 and consequently
\[
x\in \Rn \setminus B_d(x_1,(1-\frac{1}{2a^2})d(x,y)).
\]
On the other hand  \eqref{equivi} implies
\[
 \Rn \setminus B_d(x_1,(1-\frac{1}{2a^2})d(x,y))\ \subset\
 \Rn \setminus B(x_1,\frac{1}{a}(1-\frac{1}{2a^2})d(x,y))\ \subset\ \Rn \setminus B(x_1,(a+3)r),
\]
the last inclusion being true since $a>1$.

We now consider  the Perron-Wiener-Brelot solution $v$  to the
Dirichlet problem $\mathcal{L}v=0$ in $B(x_1,(a+3)r)\cap  D$, with
boundary datum a function $\phi\in C(\partial{(B(x_1,(a+3)r)\cap
D)})$, such that $0\leq \phi \leq 1$, $\phi=1$ on
$\partial{B}(x_1,(3+a)r)\cap  D$, and  $\phi=0$ on $\partial  D\cap
B(x_1,(1+a)r)$. We observe in passing that, thanks to the
assumptions on $D$, we can only say that $v$ is continuous up to the
boundary in that portion of $\partial{(B(x_1,(a+3)r)\cap  D)}$ that
is common to $\partial  D$. However such continuity is not needed to
implement Theorem \ref{T:MP} and deduce that  $0\leq v\leq 1$. We
observe that  $ D \cap B(x_1,(a+3)r)$ satisfies the outer
$\mathcal{L}$-ball condition at the point $x_o\in \partial  D$.
Applying Theorem \ref{T:Growth}  one infers  for every $y\in  D\cap
B(x_1,(a+3)r)$
\begin{equation}\label{v}
|v(y)|\ \le\  C\ \frac{d(y,\partial{ D})}{r}\ .
\end{equation}

Let $C_{ D}$ be as in \eqref{green}  and define $w(z)=C_{
D}^{-1}E(x,\beta d(x,y))G(x,z)$, where $\beta= (1-\frac{1}{2a^2} -
\frac{1}{2a})$. Since $x\notin B(x_1,(a+3)r)$, then
$\mathcal{L}w=0$ in $B(x_1,(a+3) r)\cap  D$. Observe that if $z\in
\partial B(x_1,(a+3)r)$, then
\[
d(x,z)\ \ge\  d(x,x_1) - d(z,x_1)\ \geq\ (1-\frac{1}{2a^2}) - (a+3)r\ \ge\ \beta d(x,y),
\]
from our choice of $r$ and $\beta$.  Consequently, in view of the
monotonicity of $r\to E(x,r)$ and \eqref{green},   we have that
$w\leq C_D^{-1} E(x,d(x,z)) G(x,z)\leq 1$ on
$\partial{(B(x_1,(a+3)r)\cap D)}$. By Theorem \ref{T:MP} one
concludes that
 $w(y)\leq v(y)$ in
$ D \cap B(x_1,(a+3)r)$. The estimate of $v$ established above, along with (2.1), completes the proof.

\end{proof}

It was observed in \cite[Theorem 50]{LU2} that in a  Carnot group,
by exploiting the symmetry of the Green function $G(y,x) = G(x,y)$,
one can actually improve the estimate in Theorem \ref{T:Green} as
follows
\[
G(x,y)\ \leq\ C\ d(x,y)^{-Q} d(x,\p D) d(y,\p D)\ ,\ \ x,y\in D\ ,\
x\not= y\ ,
\]
where $Q$ represents the homogeneous dimension of the group. An
analogous improvement can be obtained in the more general setting of
this paper. To see this, note that the symmetry of $G$ and the
estimate in Theorem \ref{T:Green} give for every $x,y\in D$
\begin{equation}\label{step3}
G(y,x) = \ G(x,y) \le C \frac{d(x,y)}{|B_d(x,d(x,y))|}\
d(y,\partial{ D})\ ,
\end{equation}
where $C>0$ is the constant in the statement of Theorem
\ref{T:Green}. We now argue exactly as in the case in which
\eqref{a} holds in the proof of Theorem \ref{T:Green}, except that
we now define
\[ w(z) = C^{-1} d(x,\p D)^{-1} \frac{|B_d(x,d(x,y))|}{d(x,y)} G(z,x)\
,\ \ z\in \overline{B(x_1,(a+3) r)\cap  D}\ .
\]

Using \eqref{step3} instead of \eqref{green} we reach the conclusion
that \[ w(z)\leq 1 \ , \ \ \text{for every}\ z\in
\partial{(B(x_1,(a+3)r)\cap D)}\ .
\]
Since $\mathcal L w = 0$ in $B(x_1,(a+3)r)\cap D$, by Theorem
\ref{T:MP} we conclude as before that
 $w(y)\leq v(y)$ in
$ D \cap B(x_1,(a+3)r)$. Combining this estimate with \eqref{v} we
have proved the following result.

\begin{cor}\label{C:lu}
Suppose that $D\subset \Rn$ satisfy the uniform outer $X$-ball
condition. There exists a constant $C=C(X, D)>0$ such that
\[
G(x,y)\ \le\ C\ \frac{d(x,\partial{ D}) d(y,\partial{
D})}{|B_d(x,d(x,y))|}\ ,
\]
for each $x,y\in  D$, with $x\neq y$.
\end{cor}

We now turn to estimating the horizontal gradient of the Green
function up to the boundary. The next result plays a central role in
the rest of the paper.

\begin{thrm}\label{T:XG}
Assume the uniform outer $X$-ball condition for $D\subset \Rn$.
There exists a constant $C=C(X, D)>0$ such that
\[
|XG(x,y)| \le  C\ \frac{d(x,y)}{|B_d(x,d(x,y))|},
\]
for each $x,y\in  D$, with $x\neq y$.
\end{thrm}

\begin{proof}[\textbf{Proof}]
Let $R_o$ be as in Definition \ref{D:OB}. Fix $x,y\in D$ and choose
$0<r<R_o$ such that $x\notin B_d(y,ar)\subset \overline{D}$.
Applying  Corollary \ref{C:Harnack} and \eqref{equivi} to
$G(x,\cdot)$ we obtain for every $z \in B(y,r)$
\[|XG(x,z)|\leq\ \frac{C}{r}\ G(x,z).
\]
If $d(y,\partial  D)\le 2a d(x,y)$,  we  choose $r = min\left(\frac{d(y,\partial  D)}{2a}, \frac{R_o}{2}\right)$
and then the latter inequality implies the conclusion via Theorem \ref{T:Green}.
If $d(y,\partial  D)> 2a d(x,y)$, then  keeping in mind that $G(x,\cdot)=\Gamma(x,\cdot) - h_x$, we use \eqref{gradgamma} to bound $|X\Gamma|$, and, with $r = min\left(\frac{d(y,\partial  D)}{2a}, \frac{R_o}{2}\right)$, we apply Corollary \ref{C:Harnack} and the maximum principle to obtain
\[
|Xh_x(y)|\ \leq\ \frac{C}{r}\ h_x(y)\ =\ \frac{C}{r}\ h_y(x)\ \leq\ \frac{C}{r}\ \underset{w\in \partial D}{sup}\ \Gamma(y,w)\ =\ \frac{C}{r}\ \Gamma(y,z)
\]
for some $z\in \partial{ D}$. On the other hand, one has
\[
d(x,y)\ <\ \frac{d(y,\partial{ D})}{2a}\ \leq\ \frac{d(y,z)}{2a}
\]
so that using \eqref{gradgamma} one more time
\[
\Gamma(y,z)\ \leq\ C\ \frac{1}{E(y, d(y,z))}\ \leq\ C\ \frac{1}{E(y,
2a d(x,y))} \leq\ C\ \Gamma(y,x) \leq\ C\
\frac{d(x,y)^2}{|B_d(x,d(x,y))|}\ .
\]

Replacing this inequality in the estimate for $|Xh_x(y)|$ we reach
the desired conclusion.

\end{proof}

\begin{cor}\label{C:Lipschitz}
If $D\subset \Rn$ satisfies  the uniform outer $X$-ball condition,
then for any $x_o \in D$ and every open neighborhood $U$ of
$\partial{D}$, such that $x_o\notin \overline{U}$, one has
$G(x_o,\cdot)\in \mathcal{L}^{1,\infty}(U)$. Moreover, its
$\mathcal{L}^{1,\infty}(U)$ norm depends on $D,X$ and $U$ but it is
independent of $x_o$.
\end{cor}

\paragraph{\bf Localizing the hypothesis.} It is interesting to note that one can still prove that
$G(x_o,\cdot)\in \mathcal{L}^{1,\infty}(U)$ under the weaker
hypothesis that the uniform outer $X$-ball condition be satisfied
only in a neighborhood of the characteristic set of $D$. In this
case, however, the uniform estimates in $x_o$ will be  lost. We
devote the last part of this section to the proof of this result.
Let $\Sigma = \Sigma_{D} \subset
 \partial{D}$ denote the compact set of all characteristic points.

\begin{dfn}\label{D:OBchar}
Let  $D$ be a $C^1$ domain. We say that  $D$ possesses the
\emph{uniform outer $X$-ball} in a neighborhood of $\Sigma$ if for a
given choice of an  open set $V$  containing $\Sigma$, one can find
$R_o>0$  such that   for every $Q\in V\cap
\partial D$ and  $0<r<R_o$ there exists a $X$-ball
$B(x_1,r)$ for which \eqref{ob} holds. More in general, we say that
$D$ possesses the \emph{uniform outer $X$-ball} along the set $V\cap
\partial D$ if  one can find $R_o>0$  such that   for every $x_o\in
V\cap \partial D$ and  $0<r<R_o$ there exists a $X$-ball $B(x_1,r)$
for which \eqref{ob} holds.
\end{dfn}

Our first step consists in proving "localized" versions of
Theorems \ref{T:Green}  and \ref{T:XG}.

\begin{thrm}\label{T:GreenLocal}
Let  $D\subset \Rn$ be a domain that is regular for the Dirichlet
problem. Let $P\in \partial D$ and assume that for some $\epsilon>0$
the set $D$ possesses the {uniform outer $X$-ball} along
$B_d(P,2\epsilon)\cap \partial D$. There exists a constant $C=C(X,
D)>0$ such that
\[
G(x,y) \le C \frac{d(x,y)}{|B_d(x,d(x,y))|}\ d(y,\partial{ D})
\]
for each $y\in B_d(P,\epsilon)\cap D$, and $x\in  D$, with $x\neq y$.

\end{thrm}

\begin{proof}[\textbf{Proof}]
The proof follows closely the one of Theorem \ref{T:Green} and we
will adopt the same notation as in that proof. Let $x_o$ be the
point in $\partial D$ closest to $y$. In order to apply the
arguments in the proof of Theorem \ref{T:Green} we need to show that
the set $D$ has an  outer $\mathcal{L}$-ball $B(x_1,r/a)$  at $x_o$
for every $0<r<R_0$. Given our hypothesis it suffices to show that
$x_o\in B_d(P,2\epsilon)\cap \partial D$. Observe that $d(y,x_o)\le
d(y,P) <\epsilon$, and consequently $d(P,x_o)<2 \epsilon$. Since $D$
has an outer $X$-ball $B(x_1,r/a)$  at $x_o$ for every $0<r<R_0$,
then so does the subset $B(x_1,(a+3)r)\cap D$. The rest of the proof
is a word by word repetition of the one for  Theorem \ref{T:Green}.

\end{proof}

\begin{thrm}\label{T:XGLocal}
Let  $D\subset \Rn$ be a domain that is regular for the Dirichlet
problem. Let $P\in \partial D$ and assume that for some $\epsilon>0$
the set $D$ possesses the uniform outer $X$-ball along
$B_d(P,2\epsilon)\cap \partial D$. There exists a constant $C=C(X,
D)>0$ such that
\[
|XG(x,y)| \le  C\ \frac{d(x,y)}{|B_d(x,d(x,y))|},
\]
for each $y\in B_d(P,\frac{1}{2}\epsilon)\cap D$, and $x\in  D$, with $x\neq y$.
\end{thrm}

\begin{proof}[\textbf{Proof}]
In the proof of Theorem \ref{T:XG} there is only one point where the
outer $X$-ball condition is used. Consider $y\in
B_d(P,\frac{1}{2}\epsilon) \cap D$ and assume that $d(y,\partial
D)\le d(x,y)$.  Choose $2r=\frac{d(y,\partial  D)}{a}$ and observe
that if $z\in B(y,r)$ then $d(z,y)<ar\le \epsilon/2$. Consequently
$d(z,P) \le d(z,y)+d(y,P) \le \epsilon,$ and we can apply Theorem
\ref{T:GreenLocal} to the function $G(x,z)$ concluding the proof in
the same way as before.

\end{proof}

\begin{cor}\label{C:LipschitzSigma}
Let $D\subset \Rn$  be a $C^\infty$ domain. If $D$ satisfies  the
uniform outer $X$-ball condition in a neighborhood $V$ of $\Sigma$,
then for any $x_o \in D$ and every open neighborhood $U$ of
$\partial{D}$, such that $x_o\notin \overline {U}$, one has
$||G(x_o,\cdot)||_{ \mathcal{L}^{1,\infty}(U)}\le C(x_o, D,V,U,X)$.
\end{cor}

\begin{proof}[\textbf{Proof}] Observe that $D$ is regular for the Dirichlet problem. The regularity away from the characteristic set follows by
 Theorem \ref{T:KN} and  the regularity in a neighborhood of $\Sigma$ is a consequence
of the uniform outer $X$-ball condition and of the cited results in
\cite{Citti},\cite{D}, \cite{NS} and \cite{CDG3}. Denote by $V$ the
neighborhood of $\Sigma$ where the uniform outer $X$-ball condition
holds. In view of the compactness of $\Sigma$, we have that
$W=\bigcup_{P\in \Sigma} B(P,2\epsilon) \subset V$, for some
$\epsilon >0$. We will consider also the set $A=\bigcup_{P\in
\Sigma} B(P,\frac{1}{2}\epsilon)\subset W$. In view of Theorem
\ref{T:KN}, we have that $G(x_o,\cdot)\in C^{\infty}
(\bar{D}\setminus \{A\cup \{x_0\}\})$. In particular, $G(x_o,\cdot)$
is smooth in $U\setminus A$. This implies the estimate
$||G(x_o,\cdot)||_{ \mathcal{L}^{1,\infty}(U\setminus A)} \le
C_0=C_0(x_o, D,V,X)$. To complete the proof of the corollary we
consider $y\in A$ and observe that there must be a $P\in \Sigma$
such that $y\in B(P, \frac{1}{2}\epsilon)$. Denote by $Q$ the
homogeneous dimension associated to the system $X$ in a neighborhood
of $D$. In view of Theorem \ref{T:XGLocal} we have that $|XG(x_o,y)|
\le C d(y,x_o)^{1-Q} \le C_1$, with $C_1$ depending only on $X,D$
and $U$. At this point we choose $ C(x_o, D,V,U,
X)=\min\{C_0,C_1\}$, and the proof is concluded.
\end{proof}

\section{\textbf{The subelliptic Poisson kernel and a representation formula for $\mathcal L$-harmonic functions}}\label{S:PK}

In this section we establish  a basic Poisson type representation
formula for smooth domains that satisfy the outer $X$-ball condition
in a neighborhood of the characteristic set. This results
generalizes an analogous representation formula for the Heisenberg
group $\Hn$ obtained by Lanconelli and Uguzzoni in \cite{LU} and
extended in \cite{CGN2} to groups of Heisenberg type. Consider a
domain $D$ which is regular for the Dirichlet problem. For a fixed
point $x_o\in  D$ we respectively denote by
$\Gamma(x)=\Gamma(x,x_o)$ and $G(x)=G(x,x_o)$,  the fundamental
solution of $\mathcal{L}$, and the Green function for  $ D$ and
$\mathcal{L}$ with pole at $x_o$. Recall that $G(x)=\Gamma(x)-h(x)$,
where $h$ is the unique $\mathcal{L}$-harmonic function with
boundary values $\Gamma$. We also note that due to the assumption
that $ D$ be regular,  $G, h$ are continuous in any relatively
compact subdomain  of $ \overline D\setminus \{x_o\}$. We next
consider a $C^{\infty}$ domain $\Om \subset \overline{\Om} \subset
\overline D $ containing the point $x_o$. For any $u,v\in
C^{\infty}(\overline D)$ we obtain from the divergence theorem

\begin{equation}\label{i1}
\int_\Om [u\ \mathcal{L}v - v\ \mathcal{L}u]\ dx\ =\
\sum_{j=1}^m\int_{\partial{\Om}}[v\ X_ju - u\ X_jv]<X_j,\n>\
d\sigma\ ,
\end{equation}
where $\n$ denotes the outer unit normal and $d\sigma$ the surface
measure on $\partial{\Om}$. By H\"ormander's hypoellipticity theorem
\cite{H} the function $x\to \Gamma(x_o,x)$ is in
$C^\infty(D\setminus \{x_o\})$. By Sard's theorem there exists a
sequence $s_k \nearrow \infty$ such that the sets $\{x\in \Rn\mid
\Gamma(x_o,x) = s_k\}$ are $C^\infty$ manifolds. Since by
\eqref{gradgamma} the fundamental solution has a singularity at
$x_o$, we can assume without restriction that such manifolds are
strictly contained in $\Om$. Set $\epsilon_k = F(x_o, s_k^{-1})$,
where $F(x_o, \cdot)$ is the inverse function of $E(x_o,\cdot)$
introduced in section two. The sets
$B(\epsilon_k)=B(x_o,\epsilon_k)\subset
\overline{B}(x_o,\epsilon_k)\subset \Om$ are a sequence of smooth
$X$-balls shrinking to the point $x_o$. We note explicitly that the
outer unit normal on $\partial{B(\epsilon_k)}$  is $\n
=-\frac{D\Gamma(\cdot,x_o)}{|D\Gamma(\cdot,x_o)|}$.

Applying \eqref{i1} with $v(x) = G(x)$, and $\Om$ replaced by
$\Om_{\epsilon_k}=\Om\setminus \overline B(\epsilon_k)$, where one
has $\mathcal{L}G=0$, we find
\begin{align}
\int_{\Om_{\epsilon_k}} G\ \mathcal{L}u\ dx\ &=\ \sum_{j=1}^m\int_{\partial{\Om}}[u\ X_jG -G\ X_ju]<X_j,\n>d\sigma\notag\\
&+\ \sum_{j=1}^m\int_{\partial{B}(\epsilon_k)}[G\ X_ju - u\
X_jG]<X_j,\n>d\sigma\ .\notag
\end{align}

Again the divergence theorem gives
\begin{equation}\label{i2}
\int_{B(\epsilon_k)} \mathcal{L}u\ dx\ =\ -\
\sum_{j=1}^m\int_{\partial{B(\epsilon_k)}}X_ju<X_j,\n>\ d\sigma\ .
\end{equation}

Using \eqref{i2}, and the fact that $G=\Gamma - h$, we find
\begin{align}\label{i3}
&\sum_{j=1}^m\int_{\partial{B(\epsilon_k)}}[G\ X_ju - u\ X_jG]<X_j,\n>\ d\sigma\\
&=\ \frac{1}{E(x_o,\epsilon_k)}\sum_{j=1}^m\int_{\partial{B(\epsilon_k)}}X_ju<X_j,\n>d\sigma\ -\
\sum_{j=1}^m\int_{\partial{B(\epsilon_k)}}h\ X_ju<X_j,\n>\ d\sigma\notag\\
& -\ \sum_{j=1}^m\int_{\partial{B(\epsilon_k)}}u\ X_j\Gamma<X_j,\n>\ d\sigma \ +
\ \sum_{j=1}^m\int_{\partial{B(\epsilon_k)}}u\ X_jh <X_j,\n>\ d\sigma\notag\\
&=\ -\ \frac{1}{E(x_o,\epsilon_k)}\int_{B(\epsilon_k)}\mathcal{L}u\ dx\ +\ \int_{\partial{B(\epsilon_k)}}u\ \frac{|X\Gamma|^2}{|D\Gamma|}\ d\sigma\notag\\
&+\ \sum_{j=1}^m\int_{\partial{B(\epsilon_k)}}u\ X_jh <X_j,\n>\
d\sigma\ -\ \sum_{j=1}^m\int_{\partial{B(\epsilon_k)}}h\
X_ju<X_j,\n>\ d\sigma\ .\notag
\end{align}

Using \eqref{psi} we find
\[
\int_{\partial{B(\epsilon_k)}}u\ \frac{|X\Gamma|^2}{|D\Gamma|}\
d\sigma\ =\ u(x_o)\ -\ \int_{B(\epsilon_k)}\mathcal{L}u\
\left[\Gamma -\frac{1}{E(x_o,\epsilon_k)}\right]\ dx\ .
\]

Keeping in mind that $u, h\in C^{\infty}(\overline{\Om})$, from the estimates \eqref{gradgamma} and the fact that
\[
\frac{|B(\epsilon_k)|}{E(x_o,\epsilon_k)}\ \leq\ C\ \epsilon_k^2,
\]
letting $k\to \infty$, so that $\epsilon_k \to 0$, we conclude from
\eqref{i2}, \eqref{i3},
\begin{equation}\label{rappresentazione}
u(x_o)\ =\ \sum_{j=1}^m\int_{\partial{\Om}}[G\ X_ju - u\ X_jG]\
<X_j,\n>\ d\sigma\ +\ \int_\Om G\ \mathcal{L}u\ dx\ .
\end{equation}

To summarize what we have found we introduce the following
definition.

\begin{dfn}\label{D:Xnormal}
Given a bounded open set $\Om \subset \overline \Om \subset \Rn$ of
class $C^1$, at every point $y\in \p \Om$ we let
\[
\bN^X(y) \ =\ (<\n(y),X_1(y)>,...,<\n(y),X_m(y)>)\ ,
\]
where $\n(y)$ is the outer unit normal to $\Om$ in $y$. We also set
\[
W(y)\ =\ |\bN^X(y)|\ =\ \sqrt{\sum_{j=1}^m <\n(y),X_j(y)>^2}\ .
\]
If $y\in \p \Om \setminus \Sigma$, we set
\[
\n^X(y)\ =\ \frac{\bN^X(y)}{|\bN^X(y)|}\ .
\]
One has $|\nuX(y)| = 1$ for every $y\in \p D \setminus \Sigma$.
\end{dfn}

We note explicitly from Definitions \ref{D:char} and \ref{D:Xnormal}
that one has for the characteristic set $\Sigma$ of $\Om$
\[
\Sigma\ =\ \{y\in \p \Om\mid W(y) = 0\}\ .
\]

Using the quantities introduced in this definition we can express
\eqref{rappresentazione} in the following suggestive way.

\begin{prop}\label{P:smoothP}
Let $ D\subset \Rn$ be a bounded open set with (positive) Green
function $G$ of the sub-Laplacian \eqref{sublap} and consider a
$C^2$ domain $\Om\subset \overline{\Om}\subset D$. For any $u\in
C^{\infty}( D)$ and every $x\in \Om$ one has
\begin{align*}
u(x)\ & =\  \int_{\partial{\Om}} G(x,y)<Xu(y),\bN^X(y)> d\sigma(y)\
- \int_{\p \Om} u(y)<XG(x,y),\bN^X(y)> d\sigma(y)
\\
& +\ \int_\Om G(x,y)\ \mathcal{L}u(y)\ dy\ .
\end{align*}
If moreover  $\mathcal{L}u$=$0$ in $ D$, then
\[
u(x)\ =\  \int_{\partial{\Om}} G(x,y)<Xu(y),\bN^X(y)> d\sigma(y)\ -
\int_{\p \Om} u(y)<XG(x,y),\bN^X(y)> d\sigma(y)\ .
\]
In particular, the latter equality gives for every $x\in \Om$
\[
\int_{\p \Om} <XG(x,y),\bN^X(y)> d\sigma(y) \ =\ -\ 1\ .
\]
\end{prop}

\begin{rmrk}\label{R:at}
If $u\in C^{\infty}(\overline{D})$, then we can weaken the
hypothesis on $\Omega$ and require only  $\overline{\Om}\subset
\overline{D}$ rather than $\overline{\Om}\subset  D$.
\end{rmrk}

We consider next a $C^\infty$ domain $D\subset \Rn$ satisfying the
uniform outer $X$-ball condition in a neighborhood of $\Sigma$. Our
purpose is to pass from the interior representation formula in
Proposition \ref{P:smoothP} to one on the boundary of $\partial{D}$.
The presence of characteristic points becomes important now. The
following result due to Derridj \cite[Theorem 1 ]{De1} will be
important in the sequel.

\begin{thrm}\label{T:FW}  Let $D\subset \Rn$ be a $C^\infty$ domain.
If  $\Sigma$ denotes  its characteristic set, then  $\sigma(\Sigma)=0$.
\end{thrm}

We now define two functions on $ D\times (\partial{ D}\setminus
\Sigma)$ which play a central role in the results of this paper.
They constitutes subelliptic versions of the Poisson kernel from
classical potential theory. The former function $P(x,y)$ is the
Poisson kernel for $D$ and the sub-Laplacian \eqref{sublap} with
respect to surface measure $\sigma$. The latter $K(x,y)$ is instead
the Poisson kernel with respect to the perimeter measure $\sigma_X$.
This comment will be clear after we prove Theorem \ref{T:HM} below.

\begin{dfn}[Subelliptic Poisson kernels]\label{D:PKs}
With the notation of Definition \ref{D:Xnormal}, for every $(x,y)\in
D\times (\p D \setminus \Sigma)$ we let
\begin{equation}\label{PK} P(x,y)\ =\ - <XG(x,y),\bN^X(y)>\  .
\end{equation}
We also define
\begin{equation}\label{i4}
K(x,y)\ =\ \frac{P(x,y)}{W(y)}\ =\ -\ <XG(x,y),\nuX(y)>\ .
\end{equation}
We extend the definition of $P$ and $K$ to all $ D\times
\partial{ D}$ by letting $P(x,y)=K(x,y)=0$ for any $x\in  D$ and
$y\in \Sigma$. According to Theorem \ref{T:FW} the extended
functions coincide $\sigma$-a.e. with the kernels in \eqref{PK},
\eqref{i4}.
\end{dfn}

It is important to note that if we fix $x\in D$, then in view of
Theorem \ref{T:KN} the functions $y \to P(x,y)$ and $y \to K(x,y)$
are $C^\infty$ up to $\partial{D}\setminus \Sigma$. The following
estimates, which follow immediately from \eqref{PK} and \eqref{i4},
will play an important role in the sequel. For $(x,y)\in D\times
(\partial{ D}\setminus \Sigma)$ we have
\begin{equation}\label{i6}
P(x,y)\ \leq\ W(y)\ |XG(x,y)|\ ,\ \ \  \ K(x,y)\leq\ |XG(x,y)|\ .
\end{equation}

We now introduce a new measure on $\p D$ by letting
\begin{equation}\label{dmu}
d\sigma_X\ =\ W\ d\sigma\ .
\end{equation}

We observe that since we are assuming that $D\in C^\infty$ the
density $W$ is smooth and bounded on $\p D$ and therefore
\eqref{dmu} implies that $d\sigma_X \ll d\sigma$. In view of this
observation Theorem \ref{T:FW} implies that also $\sigma_X(\Sigma) =
0$.

\begin{rmrk}\label{R:per}
We mention explicitly that the measure $d\sigma_X$ in \eqref{dmu} is
the $X$-perimeter measure $P_X(D;\cdot)$ (following De Giorgi)
concentrated on $\p D$. To explain this point we recall that for any
open set $\Om\subset \Rn$
\begin{equation}\label{XP}
P_X(D;\Om)\ =\ Var_X(\chi_D;\Om)\ ,
\end{equation}
where $Var_X$ indicates the sub-Riemannian $X$-variation  introduced
in \cite{CDG2} and also developed in \cite{GN1}. Given a bounded
$C^2$ domain $D\subset \Rn$ one obtains from \cite{CDG2} that
\begin{equation}\label{pm}
P_X(D;\Om)\ =\ \int_{\p D \cap \Om} W\ d\sigma\ .
\end{equation}
From \eqref{pm} one concludes that for every $y\in \p D$ and every
$r>0$
\begin{equation}\label{equiv}
\sigma_X(\p D \cap  B_d(y,r))\ =\ P_X(D;B_d(y,r))\ ,
\end{equation}
which explains the remark. The measure $\sigma_X = P_X(D;\cdot)$ on
$\p D$ plays a pervasive role in the analysis and geometry of
sub-Riemannian spaces, and its intrinsic properties have many deep
implications both in subelliptic pde's and in geometric measure
theory. For an account of some of these aspects we refer the reader
to \cite{DGN2}.
\end{rmrk}

\begin{prop}\label{P:One}
Let $D\subset \Rn$  be a bounded $C^{\infty}$ domain satisfying the
uniform outer $X$-ball condition in a neighborhood of its
characteristic set $\Sigma$. For every $x\in D$ we have
\[
\int_{\partial{ D}}P(x,y)d\sigma(y)\ =\ 1\ =\ \int_{\partial{
D}}K(x,y)d\sigma_X(y)\ .
\]
\end{prop}

\begin{proof}[\textbf{Proof}]
We fix $x\in D$ and recall that $\Sigma$ is a compact set. In view
of Theorem \ref{T:FW} we can choose an exhaustion of $D$ with a
family of $C^\infty$ connected open sets $\Om_k\subset
\overline{\Om}_k \subset \overline{D}$, with $\Om_k \nearrow D $ as
$k \to \infty$, such that $\partial{\Om_k} = \Gamma^1_k \cup
\Gamma^2_k$, with $\Gamma^1_k \subset
\partial{D}\setminus \Sigma$, $\Gamma^1_k \nearrow
\partial{D}$, $\sigma(\Gamma^2_k)\to 0$.
By Proposition \ref{P:smoothP} (and the remark following it) we
obtain for every $k\in \mathbb{N}$
\begin{align}\label{uno}
 -\ 1& =\ \int_{\partial{\Om_k}}<XG(x,y),\bN^X(y)> d\sigma(y)
\\
& =\ \int_{\partial{\Gamma^1_k}}<XG(x,y),\bN^X(y)>d\sigma(y)\ +\
\int_{\partial{\Gamma^2_k}}<XG(x,y),\bN^X(y)>d\sigma(y)\ . \notag
\end{align}

We now pass to the limit as $k\to \infty$ in the above integrals.
Using Corollary \ref{C:LipschitzSigma} and $\sigma(\Gamma^2_k)\to
0$, we infer that
\[
\underset{k\to \infty}{\lim}\
\int_{\partial{\Gamma^2_k}}<XG(x,y),\bN^X(y)>d\sigma(y)\ =\  0\ .
\]

Theorem \ref{T:KN}, Corollary \ref{C:LipschitzSigma}, and the fact
that $\Gamma^1_k \nearrow \partial{D}$, allow to use dominated
convergence and obtain
\[
\underset{k\to \infty}{\lim}\
\int_{\partial{\Gamma^1_k}}<XG(x,y),\bN^X(y)>d\sigma(y)\ = \
\int_{\partial{D}}<XG(x,y),\bN^X(y)>d\sigma(y)\ .
\]

In conclusion, we have found
\[
-\ 1\ =\ \int_{\partial{D}}<XG(x,y),\bN^X(y)>d\sigma(y)\ ,
\]
which, in view of \eqref{PK}, proves the first identity. To
establish the second identity we return to \eqref{uno}, which in
view of \eqref{i4},  \eqref{dmu} we can rewrite as follows
\begin{align*}
1  & =\ - \int_{\partial{\Gamma^1_k}}<XG(x,y),\n^X(y)>d\sigma_X(y)\
-\ \int_{\partial{\Gamma^2_k}}<XG(x,y),\bN^X(y)>d\sigma(y)
\\
& =\ \int_{\partial{\Gamma^1_k}} K(x,y)\ d\sigma_X(y)\ -\
\int_{\partial{\Gamma^2_k}}<XG(x,y),\bN^X(y)>d\sigma(y)\ .
\end{align*}

Since as we have observed $d\sigma_X \ll d\sigma$, in view of the
second estimate $K(x,y)\leq |XG(x,y)|$ in \eqref{i6}, we can again
use Theorem \ref{T:KN}, Corollary \ref{C:LipschitzSigma} and
dominated convergence (with respect to $\sigma_X$) to conclude that
\[
\underset{k\to \infty}{\lim}\ \int_{\partial{\Gamma^1_k}} K(x,y)\
d\sigma_X(y)\ = \ \int_{\partial{D}}K(x,y)\ d\sigma_X(y)\ .
\]

This completes the proof.

\end{proof}

\begin{thrm}\label{T:PK}
Let $ D$ satisfy the assumptions in Proposition \ref{P:One}. If
$\phi\in C^\infty(\partial{D})$ assumes a single constant value in a
neighborhood of $\Sigma$,  then $H^{D}_{\phi} \in
\mathcal{L}^{1,\infty}( D)$. Furthermore, if for $\phi \in
C(\partial{D})$ we have $H^D_\phi \in \mathcal{L}^{1,\infty}(D)$,
then
\[
H^{ D}_{\phi}(x)\ =\ \int_{\partial{ D}}P(x,y)\ \phi(y)\ d\sigma(y)\
=\ \int_{\partial{ D}}K(x,y)\ \phi(y)\ d\sigma_X(y)\ , \quad \quad \
x\in D\ .
\]
\end{thrm}

\begin{proof}[\textbf{Proof}]
We start with the proof of the regularity result. Let $\phi$ be as
in the first part of the statement. We mention explicitly that, by
definition, $\phi$ is $C^\infty$ in a neighborhood of $\p D$.
 Denote by $U$  a neighborhood of $\Sigma$ in which the function $\phi$ is constant
and along which the domain $D$ satisfies the uniform outer $X$-ball
condition. As in the proof of Corollary \ref{C:LipschitzSigma}, we
can assume that $U=\bigcup_{P\in \Sigma} B_d(P,\epsilon)$, for some
$\epsilon=\epsilon(U,X)>0$. If we denote by $R_0$ the constant
involved in the definition of the uniform outer $X$-ball (see
Definition \ref{D:OBchar}), then we can always select a smaller
constant so that $\epsilon=2aR_0$ (here $a>1$ is the constant from
\eqref{equivi}).
  In view of Proposition \ref{P:One} we can assume without loss of generality that $\phi$ vanishes in a neighborhood of $\Sigma$ and $\underset{\partial{D}}{max}|\phi|=1$.
We want to show that the horizontal gradient of $H^D_{\phi}$ is in
$L^{\infty}$ in such neighborhood. By Theorem \ref{T:KN} the
conclusion $H^D_{\phi} \in \mathcal{L}^{1,\infty}(D)$ will follow.
Fix $x_o\in \Sigma$, and $0<r<R_o$, where $R_o$ is as in Definition
\ref{D:OB}. Theorem \ref{T:Growth} implies
\begin{equation}\label{i7}
|H^D_{\phi}(y)|\leq\ C\ \frac{d(y,x_o)}{r}
\end{equation}
for every $y\in  D$. Let now $x\in B(x_o,r/2)\cap  D$ and consider
the metric ball $B_d(x,a^{-1}\tau)\subset B(x,\tau)$, see
\eqref{equivi}, where $\tau=\frac{d(x,\partial{ D})}{4}$. Corollary
\ref{C:Harnack} implies
\begin{equation}\label{i8}
|XH^D_{\phi}(x)|\leq\ \frac{C}{d(x,\partial{ D})}\ H^D_{\phi}(x).
\end{equation}

Pick $P\in \partial{ D}$ such that $d(x,P)=d(x,\partial{ D})$.
Observe that $d(P,\Sigma)\le d(P,x_o)\le d(P,x)+d(x,x_o) \le
2d(x,x_o) \le aR_0 =\epsilon/2$. In particular we can apply once
more Theorem \ref{T:Growth}, and obtain \eqref{i7} with $P$ in place
of $x_o$. Arguing in this way we find
\[
H^D_{\phi}(x)\leq\ C\ \frac{d(x,P)}{r}\ =\ C\ \frac{d(x,\partial{ D})}{r}.
\]

The latter inequality and \eqref{i8} imply
\[
|XH^D_{\phi}(x)|\ \leq \ \frac{C}{r}.
\]

This proves that $|XH^D_{\phi}|\in L^{\infty}(B(x_o,r/2)\cap  D)$.
To establish the second part of the theorem, we take a function
$\phi \in C(\partial{D})$ for which $H^D_\phi \in
\mathcal{L}^{1,\infty}(D)$. We fix $x\in  D$ and consider the
sequence of $C^{\infty}$ domains $\Om_k$ as in the proof of
Proposition \ref{P:One}.  Proposition \ref{P:smoothP} gives
\begin{equation}\label{i9}
H^D_\phi(x)\ =\  \int_{\partial{\Om}_k} G(x,y)
<X(H^D_\phi)(y),\bN^X(y)>  d\sigma(y)\ -\
\int_{\partial{\Om}_k}H^D_\phi(y) <XG(x,y),\bN^X(y)> d\sigma(y)\ .
\end{equation}

At this point the conclusion follows along the lines of the proof of Proposition \ref{P:One}.

\end{proof}

\begin{prop}\label{P:Positivity}
Let $ D$ be a $C^{\infty}$ domain. i) If $D$  satisfies  the uniform
outer $X$-ball condition in a neighborhood of $\Sigma$, then
$P(x,y)\geq 0$ and $K(x,y) \ge 0$ for each $(x,y)\in
D\times\partial{D}$; ii) If $D$ satisfies  the uniform outer
$X$-ball condition, then there exists a constant $C_{D}>0$ such that
for $(x,y)\in  D\times \partial{D}$
\[
0\ \leq\  P(x,y)\ \leq\ C_{D}\ W(y)\ \frac{d(x,y)}{|B_d(x,d(x,y))|},
\quad \ 0 \ \leq\ K(x,y)\ \leq \ C_{D}\
\frac{d(x,y)}{|B_d(x,d(x,y))|}\ .
\]
In particular, if we fix $x\in D$, then for any open set $U$
containing $\p D$, such that $x\notin \overline U$, one has
$K(x,\cdot)\in L^\infty(\overline D \cap U)$.
\end{prop}

\begin{proof}[\textbf{Proof}]
We start with the proof of part (i). We argue by contradiction.  If
for some $x\in  D$ and $x_o\in \partial{ D}$ we had $P(x,x_o)=\
\alpha\ <0$, then $x_o\notin \Sigma$. By Theorem \ref{T:KN} there
exists a sufficiently small $r>0$  such that $P(x,x')\leq \alpha/2$
for every $x'\in B(x_o,2r)\cap \partial{D}$. We can also assume that
$d(x_o,\Sigma) > 2r$.  We now choose $\phi\in C^\infty(\partial{D})$
such that $0\leq \phi \leq 1$, $\phi\equiv 1$ on $B(x_o,r)\cap
\partial{D}$ and $\phi\equiv 0$ outside $B(x_o,3r/2)\cap \partial{D}$.
Theorem \ref{T:MP} implies $H^{ D}_{\phi}\geq 0$ in $D$. By the
Harnack inequality we must have $H^{D}_{\phi}(x)>0$. On the other
hand, Theorem \ref{T:PK} gives
\[
H^{D}_{\phi}(x)\ \leq \ \frac{\alpha}{2}\ \int_{B(x_o,3r/2)\cap
\partial{D}}\ \phi(y)\ d\sigma(y)\ \leq\ 0\ ,
\]
which gives a contradiction. The proof of part (ii) is an immediate
consequence of \eqref{i6} and of Theorem \ref{T:XG}.  The estimate
for $K(x,y)$ follows from \eqref{i4} and from the one for $P(x,y)$.

\end{proof}

We now fix $x\in D$. For every $\sigma$-measurable $E\subset
\partial{ D}$ we set
\[
\nu^x(E)\ =\ \int_{E}\ K(x,y)\ d\sigma_X(y)\ .
\]

According to Proposition \ref{P:Positivity}, $d\nu^x$ defines a
Borel measure on $\partial{ D}$. Using Theorems \ref{T:FW} and
\ref{T:PK} we can now establish the main result of this section.

\begin{thrm}\label{T:HM}
Let $D\subset \Rn$ be a $C^{\infty}$ domain possessing the uniform
outer $X$-ball condition in a neighborhood of  the characteristic
set $\Sigma$. For every $x\in D$, we have $\omega^x\ =\ \nu^x$,
i.e., for every $\phi\in C(\partial{D})$ one has
\[
H^{D}_{\phi}(x)\ =\ \int_{\partial{D}}\ \phi(y)\ K(x,y)\
d\sigma_X(y) \ =\ \int_{\partial{D}}\ \phi(y)\ P(x,y)\ d\sigma(y)\ ,
\quad   x\in D\ .
\]
In particular, $d\omega^x$ is absolutely continuous with respect to
$d\sigma_X$ and $d\sigma$, and for every $(x,y)\in D \times
\partial{D}$ one has
\begin{equation}\label{kernels}
\frac{d\omega^x}{d\sigma_X}(y) \ =\ K(x,y),\ \quad \quad
\frac{d\omega^x}{d\sigma}(y) \ =\ P(x,y) \ .
\end{equation}
\end{thrm}

\begin{proof}[\textbf{Proof.}]

We begin with proving \eqref{kernels}. Let $F\subset\partial{D}$ be
a Borel set. If $F=\partial{D}$ then the result follows from
Proposition \ref{P:One}.  We now consider the case when the
inclusion $F\subset
\partial{D}$ is strict.  Choose $\eps >0$. Since both $K(x,y)$ and
$W(y)$ are bounded,  there exists open sets $E_\eps,
F_\eps\subset\partial D$ such that $F\subset F_\eps\subset
\overline{F}_\eps \subset E_\eps$, and $\nu^x(E_\eps\setminus F) <
\eps/2$.   Theorem \ref{T:FW} guarantees the existence of open sets
$\Sigma_\eps,U_\eps$ such that $\Sigma\subset
\Sigma_\eps\subset\overline{\Sigma_\eps}\subset U_\eps$ and
$\nu^x(U_\eps) < \eps/2$.  We now choose a function $\phi\in
C^\infty_o(\partial D)$ and $0\leq \phi \leq 1$ with $\phi\equiv 1$
on $U_\eps$ and $\nu^x(supp\,\phi) < \frac{3}{4}\eps$.  We have
\begin{align}\label{4.16?}
\omega^x(U_\eps) &= \int_{U_\eps}  \,d\omega^x(y) \leq \int_{\partial D}\phi(y)\,d\omega^x(y) = H^{D}_{\phi}(x) \\
\text{(by Theorem \ref{T:PK}) } & = \int_{\partial
D}\phi(y)\,K(x,y)d\sigma_X(y) \leq \nu^x(supp\,\phi) <
\frac{3}{4}\eps\ .\notag
\end{align}

Let now $\psi_o,\psi_1 \in C^\infty_o(\partial D)$ such that $0 \leq \psi_o,\psi_1 \leq 1$ and
\begin{align*}
\psi_o \equiv 1 \,\text{in } &\partial D \setminus U_\eps, \quad \psi_o \equiv 0 \,\text{in } \Sigma_\eps, \\
\psi_1 \equiv 1 \,\text{in } &F, \quad \psi_1 \equiv 0 \,\text{in }
\partial D\setminus E_\eps\ .
\end{align*}

One has
\begin{align*}
\omega^x(F) &\ \leq\ \omega^x(U_\eps) + \omega^x(F\setminus U_\eps) \quad \text{(by \eqref{4.16?} )}\\
&\ \leq\ \frac{3}{4}\eps + \int_{\partial D} \psi_o(y)\psi_1(y) \,d\omega^x(y) \quad \\
&\ =\ \frac{3}{4}\eps + H^{D}_{\psi_o\psi_1}(x) \quad \text{(by Theorem \ref{T:PK})} \\
&\ =\ \frac{3}{4}\eps + \int_{\partial D} \psi_o(y)\psi_1(y) K(x,y)\,d\sigma_X(y) \leq \frac{3}{4}\eps + \nu^x(E_\eps) \\
&\ =\ \frac{3}{4}\eps + \nu^x(F) + \nu^x(E_\eps\setminus F) <
\nu^x(F) + \frac{7}{4}\eps\ .
\end{align*}

Since $\eps>0$ is arbitrary, we infer that $\omega^x(F) \leq
\nu^x(F)$.  If we repeat the same argument with $E_\eps\setminus F$
playing the role of the set $F$, we can prove
$\omega^x(E_\eps\setminus F) \leq \nu^x(E_\eps\setminus F)$. This
allows to exchange the role of $\omega^x$ and $\nu^x$ in the
computations above and conclude $\nu^x(F) \leq \omega^x(F)$.

To complete the proof of the theorem we now use \eqref{kernels}.
From the definition of harmonic measure we know that for each
$\phi\in C(\partial{D})$ and $x\in D$ we have
\begin{equation}\label{i11}
H^D_\phi(x) = \int_{\partial{D}} \phi(y) d\omega^x(y).
\end{equation}

On the other hand \eqref{kernels} yields $d\omega^x(y) = K(x,y)
d\sigma_X(y)$. If we substitute the latter in \eqref{i11} we reach
the conclusion.

\end{proof}

\section{\textbf{Reverse H\"older inequalities for the Poisson kernel}}\label{S:RHI}

This section is devoted to proving the main results of this paper,
namely  Theorems \ref{T:RHGeneral}, \ref{T:Dirichlet},
\ref{T:RHGeneral-P} and \ref{T:Dirichletsigma}. In the course of the
proofs we will need some basic results about $NTA_X$ domains from
the paper \cite{CG}. We begin by recalling the relevant definitions.

\begin{dfn}\label{D:NTA}
We say that a connected, bounded open set $ D\subset \Rn$ is a
non-tangentially accessible domain with respect to the system $X =
\{X_1,...,X_m\}$ ($NTA_X$ domain, hereafter) if there exist $M$,
$r_o>0$ for which:
\begin{itemize}
\item[(i)] (Interior corkscrew condition) For any $x_o\in\partial D$ and $r\leq r_o$
there exists $A_r(x_o)\in D$ such that
$\frac{r}M<d(A_r(x_o),x_o)\leq r$ and $d(A_r(x_o),\partial
D)>\frac{r}M$. (This implies that $B_d(A_r(x_o),\frac{r}{2M})$ is
$(3M,X)$-nontangential.)
\item[(ii)] (Exterior corkscrew condition) $ D^c=\Rn\setminus  D$ satisfies
property (i).
\item[(iii)] (Harnack chain condition) There exists $C(M)>0$ such that for any $\epsilon>0$ and $x,y\in D$ such
that $d(x,\partial D)>\epsilon$, $d(y,\partial D)>\epsilon$, and $d(x,y)<C\epsilon$, there exists a Harnack
chain joining $x$ to $y$ whose length depends on $C$ but not on $\epsilon$.
\end{itemize}
\end{dfn}

We note the following lemma which will prove useful in the sequel
and which follows directly from Definition \ref{D:NTA}.

\begin{lemma}\label{L:cork}
Let $D\subset \Rn$ be $NTA_X$ domain, then there exist constants $C,
R_1$ depending on the $NTA_X$ parameters of $D$ such that for every
$y\in \p D$ and every $0<r<R_1$ one has
\[
C \ |B_d(y,r)|\ \leq\ \min \{|D\cap B_d(y,r)|,|D^c \cap B_d(y,r)|\}\
\leq\ C^{-1}\ |B_d(y,r)|\ .
\]
In particular, every $NTA_X$ domain has positive density at every
boundary point and therefore it is regular for the Dirichlet problem
(see Definition \ref{D:posden}, Proposition \ref{P:posden}, and
Theorem \ref{T:thin}).
\end{lemma}

In the sequel, for $y\in \partial{D}$ and $r>0$ we denote by
\[
\Delta(y,r)\ =\ \p D\ \cap\ B_d(y,r) \] the surface metric ball
centered at $y$ with radius $r$. We next prove a basic
non-degeneracy property of the horizontal perimeter measure
$d\sigma_X$ in \eqref{dmu}.

\begin{thrm}\label{T:ndeg}
Let $D\subset \Rn$ be a $NTA_X$ domain of class $C^2$, then there
exist $C^*, R_1>0$ depending on $D$, $X$ and on the $NTA_X$
parameters of $D$ such that for every $y\in \p D$ and every
$0<r<R_1$
\[
\sigma_X(\Delta(y,r))\ \geq \ C^*\ \frac{|B_d(y,r)|}{r}\ .
\]
In particular, $\sigma_X$ is lower $1$-Ahlfors according to
\cite{DGN2} and $\sigma_X(\Delta(y,r))>0$.
\end{thrm}

\begin{proof}[\textbf{Proof}]
According to (I) in Theorem 1.15 in \cite{GN1} every metric ball
$B_d(y,r)$ is a $PS_X$ (Poincar\'e-Sobolev) domain with respect to
the system $X$. We can thus apply the isoperimetric inequality
Theorem 1.18 in \cite{GN1} to infer the existence of $R_1>0$ such
that for every $y\in \p D$ and every $0<r<R_1$
\[
\min \{|D \cap B_d(y,r)|, |D^c \cap B_d(y,r)|\}^{\frac{Q-1}{Q}}
\leq\ C_{iso}\ \frac{diam\ B_d(y,r)}{|B_d(y,r)|^{\frac{1}{Q}}}\
P_X(D;B_d(y,r))\ ,
\]
where $Q$ is the homogeneous dimension of a fixed bounded set $U$
containing $\overline D$. On the other hand, every $NTA_X$ domain is
a $PS_X$ domain. We can thus combine the latter inequality with
\eqref{equiv} and Lemma \ref{L:cork} to finally obtain
\[
\sigma_X(\Delta(y,r))\ \geq\ C^*\ \frac{|B_d(y,r)|}{r}\ .
\]

This proves the theorem.

\end{proof}

\begin{cor}\label{C:ahlfors}
Let $D\subset \Rn$ be a $NTA_X$ domain of class $C^2$ satisfying the
upper $1$-Ahlfors assumption in $iv)$ of Definition \ref{D:ADPX}.
Then the measure $\sigma_X$ is $1$-Ahlfors, in the sense that there
exist $\tilde{A}, R_1>0$ depending on the $NTA_X$ parameters of $D$
and on $A>0$ in $iv)$, such that for every $y\in \p D$, and every
$0<r<R_1$, one has
\begin{equation}\label{ahlfors}
\tilde A\ \frac{|B_d(y,r)|}{r}\ \leq\ \sigma_X(\Delta(y,r))\ \leq\
\tilde{A}^{-1}\ \frac{|B_d(y,r)|}{r}\ .
\end{equation}
In particular, the measure $\sigma_X$ is doubling, i.e., there
exists $C>0$ depending on $\tilde A$ and on the constant $C_1$ in
\eqref{dc}, such that
\begin{equation}\label{deltadoubling}
\sigma_X(\Delta(y,2r))\ \leq\ C\ \sigma_X(\Delta(y,r))\ .
\end{equation}
for every $y\in \p D$ and $0<r<R_1$.
\end{cor}

\begin{proof}[\textbf{Proof}]
According to Theorem \ref{T:ndeg} the measure $\sigma_X$ is lower
$1$-Ahlfors. Since by $iv)$ of Definition \ref{D:ADPX} it is also
upper $1$-Ahlfors, the conclusion \eqref{ahlfors} follows. From the
latter and the doubling condition \eqref{dc} for the metric balls,
we reach the desired conclusion \eqref{deltadoubling}.

\end{proof}

The following results from \cite{CG} play a fundamental role in this
paper.

\begin{thrm}\label{T:BdryE}
Let $ D\subset \Rn$ be a $NTA_X$ domain with relative parameters $M,
r_o$. There exists a constant $C>0$, depending only on $X$ and on
the $NTA_X$ parameters of $D$, $M$ and $r_o$, such that for every
$x_o\in \partial{D}$ one has
\[
\omega^{A_r(x_o)}(\Delta(x_o,r))\ \geq\ C\ .
\]
\end{thrm}

\begin{thrm}[Doubling condition for $\mathcal{L}$-harmonic measure]\label{T:Doubling Condition}
Consider a $NTA_X$ domain $ D\subset \Rn$  with relative parameters
$M, r_o$. Let $x_o\in\partial{D}$ and $r\leq r_o$. There exist
$C>0$, depending on $X, M$ and $r_o$, such that
$$
\omega^x(\Delta(x_o,2r))\leq C\omega^x(\Delta(x_o,r))
$$
for any $x\in D\setminus B_d(x_o,Mr)$.
\end{thrm}

\begin{thrm}[Comparison theorem]\label{T:GlobalComparison}
Let $ D\subset \Rn$ be a $X-NTA$ domain with relative parameters $M,
r_o$. Let $x_o\in \partial  D$ and $0<r<\frac{r_o}{M}$. If $u, v$
are $\mathcal{L}$-harmonic functions in $ D$, which vanish
continuously on $\partial{ D}\setminus \Delta(x_o,2r)$, then for
every $x\in D\setminus B_d(x_o,Mr)$ one has
\[
C\ \frac{u(A_r(x_o))}{v(A_r(x_o))}\ \leq\ \frac{u(x)}{v(x)}\ \leq\
C^{-1}\ \frac{u(A_r(x_o))}{v(A_r(x_o))}
\]
for some constant $C>0$ depending only on $X, M$ and $r_o$.
\end{thrm}\

For any $y\in\partial\Om$ and $\alpha > 0$ a nontangential region at
$y$ is defined by
\[
\Gamma_\alpha(y) = \{x\in\Om\,|\, d(x,y) \leq
(1+\alpha)d(x,\partial\Om)\}\ .
\]

Given a function $u$ the $\alpha$-nontangential maximal function of
$u$ at $y\in \p D$ is defined by
\[
N_\alpha(u)(y) = \underset{x\in\Gamma_\alpha(y)}{sup}\,|u(x)|\ .
\]

\begin{thrm}\label{T:Representation}
Let $D\subset \Rn$ be a $NTA_X$ domain. Given a point $x_1\in D$,
let $f\in L^1(\partial{D},d\omega^{x_1})$ and define
$$
u(x)=\int_{\partial{D}} f(y)d\omega^x(y),\qquad x\in D\ .
$$
Then, $u$ is $\mathcal{L}$-harmonic in $D$, and:
\begin{itemize}
\item[(i)] $N_\alpha(u)(y)\leq CM_{\omega^{x_1}}(f)(y)$, $y\in\partial D$;
\item[(ii)] $u$ converges non-tangentially a.e.\ $(d\omega^{x_1})$ to $f$.
\end{itemize}
\end{thrm}

Theorem \ref{T:GlobalComparison} has the following important consequence.

\begin{thrm}\label{T:Komega}
Let $ D\subset \Rn$ be a $ADP_X$ domain, and let $K(\cdot,\cdot)$ be
the Poisson Kernel defined in \eqref{i4}. There exists $r_1>0$,
depending on $M$ and $r_o$ , and a constant $C=C(X,M,r_o,R_o)>0$,
such that given $x_o\in\partial{D}$, for every $x\in D\setminus
B_d(x_o,Mr)$ and every $0<r<r_1$ one can find $E_{x_o,x,r}\subset
\Delta(x_o,r)$, with $\sigma_X( E_{x_o,x,r}) = 0$, for which
\[
K(x,y)\ \leq\ C\ K(A_r(x_o),y)\ \omega^x(\Delta(x_o,r))
\]
for every $y\in \Delta(x_o,r)\setminus E_{x_o,x,r}$.
\end{thrm}

\begin{proof}[\textbf{Proof}]
Let $x_o\in
\partial{ D}$. For each $y\in \Delta(x_o,r)$ and $0<s<r/2$ set
\[
u(x)\ =\ \omega^x(\Delta(y,s)), \ \ \ \ v(x)\ =\
\omega^x(\Delta(x_o,r/2))\ .
\]

The functions $u$ and $v$ are $\mathcal{L}$-harmonic in $D$ and
vanish continuously on $\partial{D}\setminus \Delta(x_o,2r)$.
Theorem \ref{T:GlobalComparison} gives
\begin{equation}\label{pk1}
\frac{\omega^x(\Delta(y,s))}{\omega^x(\Delta(x_o,r/2))}\ \leq\  C\
\frac{\omega^{A_r(x_o)}(\Delta(y,s))}{\omega^{A_r(x_o)}(\Delta(x_o,r/2))}
\end{equation}
for every $x\in  D\setminus B(x_o,Mr)$. Applying \eqref{pk1} we thus
find
\begin{equation}\label{pk2}
\frac{\omega^x(\Delta(y,s))}{\omega^x(\Delta(x_o,r/2))}\ \leq\  C\
\frac{\omega^{A_r(x_o)}(\Delta(y,s))}{\omega^{A_r(x_o)}(\Delta(x_o,r/2))}\
.
\end{equation}

Upon dividing by $\sigma_X(\Delta(y,s))$ in \eqref{pk2} (observe
that in view of Theorem \ref{T:ndeg} the $\sigma_X$ measure of any
surface ball $\Delta(y,s)$ is strictly positive), one concludes
\begin{equation}\label{pk3}
\frac{\omega^x(\Delta(y,s))}{\sigma_X(\Delta(y,s))}\ \leq\  C\
\frac{\omega^{A_r(x_o)}(\Delta(y,s))}{\sigma_X(\Delta(y,s))}\
\frac{\omega^x(\Delta(x_o,r/2))}{\omega^{A_r(x_o)}(\Delta(x_o,r/2))}\
.
\end{equation}

Using Theorem \ref{T:BdryE} in the right-hand side of \eqref{pk3} we
conclude
\begin{equation}\label{pk4}
\frac{\omega^x(\Delta(y,s))}{\sigma_X(\Delta(y,s))}\ \leq\  C\
\frac{\omega^{A_r(x_o)}(\Delta(y,s))}{\sigma_X(\Delta(y,s))}\
\omega^x(\Delta(x_o,r))\ .
\end{equation}

We now observe that \eqref{deltadoubling} in Corollary
\ref{C:ahlfors} allows to obtain a Vitali covering theorem and
differentiate the measure $\omega^x$ with respect to the horizontal
perimeter measure $\sigma_X$. This means that for $\sigma_X$-a.e.
$y\in \Delta(x_o,r)$ the limit $\underset{s\to 0}{\lim}
\frac{\omega^x(\Delta(y,s))}{\sigma_X(\Delta(y,s))}$ exists and
equals $\frac{d\omega^x}{d\sigma_X}(y)$. This being said, passing to
the limit as $s\to 0^+$ in \eqref{pk4} we obtain for $\sigma_X$-a.e.
$y\in \Delta(x_o,r)$
\[ \frac{d\omega^x}{d\sigma_X}(y)\ \leq\ C\
\frac{d\omega^{A_r(x_o)}}{d\sigma_X}(y)\ \omega^x(\Delta(x_o,r))\ .
\]

Since by \eqref{kernels} in Theorem \ref{T:HM} we know that
$\frac{d\omega^x}{d\sigma_X}(y) = K(x,y)$,
$\frac{d\omega^{A_r(x_o)}}{d\sigma_X}(y) = K(A_r(x_o),y)$, we have
reached the desired conclusion. We observe in passing that the
exceptional set here depends on $x$ and on $A_r(x_o)$, but this fact
will be inconsequential to us since we plan to integrate with
respect to $\sigma_X$ the above inequality on the surface ball
$\Delta(x_o,r)$.

\end{proof}

We now turn to the

\begin{proof}[\textbf{Proof of Theorem \ref{T:RHGeneral}}]
We fix $p>1$, $x_o\in\partial D$ and $x_1\in D$.  Let $R_1$ be the
minimum of the constants appearing in Definitions \ref{D:OB},
\ref{D:NTA}, and in Theorem \ref{T:Komega}. Moreover, the constant
$R_1$ should be chosen so small that $d(x_1,x_o) > MR_1$. Let $0< r
< R_1$. If $A_r(x_o)$ is a corkscrew for $x_o$, then by the
definition of a corkscrew, the triangle inequality and (2.3) it is
easy to see that we have for all $y\in\Delta(x_o,r)$
\begin{equation}\label{6.5}
d(A_r(x_o),y) \sim Cr \quad\text{and}\quad |B_d(x_o,r)| \leq
C|B_d(A_r(x_o),d(A_r(x_o),y))|.
\end{equation}

Now we have
\begin{align*}
&\left(\frac{1}{\sigma_X(\Delta(x_o,r))} \ \int_{\Delta(x_o,r)}\
K(x_1,y)^p\ d\sigma_X(y)\right)^\frac{1}{p}
\quad\quad\text{(by \eqref{kernels})} \\
&\quad =\left(\frac{1}{\sigma_X(\Delta(x_o,r))}\
\int_{\Delta(x_o,r)}\ K(x_1,y)^{p-1}\
d\omega^{x_1}(y)\right)^\frac{1}{p}
\quad\text{(by Theorem \ref{T:Komega})} \\
&\quad \leq\ C\
\left(\frac{\omega^{x_1}(\Delta(x_o,r))^{p-1}}{\sigma_X(\Delta(x_o,r))}\
\int_{\Delta(x_o,r)}\ K(A_r(x_o),y)^{p-1}\ d\omega^{x_1}(y)\right)^\frac{1}{p} \quad\text{(by \eqref{i6})} \\
&\quad \leq\ \ C\
\left(\frac{\omega^{x_1}(\Delta(x_o,r))^{p-1}}{\sigma_X(\Delta(x_o,r))}\
\int_{\Delta(x_o,r)}\ |XG(A_r(x_o),y)|^{p-1}\
 d\omega^{x_1}(y) \right)^\frac{1}{p}\quad\text{(by Theorem \ref{T:XG})} \\
&\quad \leq  C\
\left(\frac{\omega^{x_1}(\Delta(x_o,r))^{p-1}}{\sigma_X(\Delta(x_o,r))}\
\int_{\Delta(x_o,r)} \
\left(\frac{d(A_r(x_o),y)}{|B_d(A_r(x_o),d(A_r(x_o),y))|}\right)^{p-1}\ d\omega^{x_1}(y)\right)^\frac{1}{p}\quad\text{(by \eqref{6.5})} \\
&\quad \leq C\
\left(\frac{\omega^{x_1}(\Delta(x_o,r))^{p-1}}{\sigma_X(\Delta(x_o,r))}\
\left(\frac{r}{|B_d(x_o,r)|}\right)^{p-1}
\int_{\Delta(x_o,r)}\ \ d\omega^{x_1}(y)\right)^\frac{1}{p}\quad\text{(by $iv)$ in Definition \ref{D:ADPX})} \\
&\quad
\leq \frac{C}{\sigma_X(\Delta(x_o,r))}\int_{\Delta(x_o,r)} \ d\omega^{x_1}(y) \quad\text{(by \eqref{kernels})}\\
&\quad = \frac{C}{\sigma_X(\Delta(x_o,r))}\int_{\Delta(x_o,r)}
K(x_1,y) d\sigma_X(y)\ .
\end{align*}

This concludes the proof of the reverse H\"older inequality.
Regarding absolute continuity, we already know from \eqref{kernels}
that $d\omega^{x_1}$ is absolutely continuous with respect to
$d\sigma_X$. To prove that $d\sigma_X$ is absolutely continuous with
respect to $d\omega^{x_1}$ we only need to observe that the reverse
H\"older inequality for $K$ established above and the doubling
property for $d\sigma_X$ from \eqref{deltadoubling} in Corollary
\ref{C:ahlfors} allow us to invoke Lemma 5 from \cite{CF}.

\end{proof}

We next establish a reverse H\"older inequality for the kernel
$P(x,y)$ defined in \eqref{PK}. The main trust of this result is
that, under a certain balanced-degeneracy assumption on the surface
measure $\sigma$ of $\p D$, it implies the mutual absolute
continuity of $\mathcal L$-harmonic measure and surface measure.
Given the fact that, as we have explained in the introduction,
surface measure is not the natural measure in the subelliptic
Dirichlet problem, being able to isolate a condition which
guarantees such mutual absolute continuity has some evident
important consequences. To state the main result we modify the class
of $ADP_X$ domains in Definition \ref{D:ADPX}. Specifically, we pose
the following

\begin{dfn}\label{D:sADPX}
Given a system $X = \{X_1,...,X_m\}$ of smooth vector fields
satisfying \eqref{frc}, we say that a connected bounded open set $
D\subset \Rn$ is \emph{$\sigma$-admissible for the Dirichlet problem
\eqref{DP}}, or simply $\sigma-ADP_X$, if:

i) $D$ is of class $C^\infty$;

ii) $D$ is non-tangentially accessible ($NTA_X$) with respect to the
Carnot-Caratheodory metric associated to the system
$\{X_1,...,X_m\}$ (see Definition \ref{D:NTA});

iii) $D$ satisfies a uniform tangent outer $X$-ball condition (see
Definition \ref{D:OB});

iv) There exist $B, R_o>0$ depending on $X$ and $D$ such that for
every $x_o\in \p D$ and $0<r<R_o$ one has
\[
\left(\underset{y\in\Delta(x_o,r)}{max}\;W(y)\right)\,\sigma(\Delta(x_o,r))\
\leq\ B\ \frac{|B_d(x_o,r)|}{r}\ .
\]
\end{dfn}

We note that Definitions \ref{D:ADPX} and \ref{D:sADPX} differ only
in part $iv)$. Also, \eqref{dmu} gives \[ \sigma_X(\Delta(x_o,r))\
=\ \int_{\Delta(x_o,r)} W(y) d\sigma(y)\ \leq\
\left(\underset{y\in\Delta(x_o,r)}{max}\;W(y)\right)\,\sigma(\Delta(x_o,r))\
. \]

This observation shows that  \[ \sigma-ADP_X\ \subset ADP_X\ .
\]

The reason for which we have referred to the new assumption on
$\sigma$ as a balanced-degeneracy condition is that, as we have seen
in the introduction the measure $\sigma$ badly degenerates on the
characteristic set $\Sigma$. On the other hand, the angle function
$W$ vanishes on $\Sigma$, thus balancing such degeneracy.

\begin{proof}[\textbf{Proof of Theorem \ref{T:RHGeneral-P}}]
The relevant reverse H\"older inequality for $P(x_1,\cdot)$ is
proved starting from the second identity $d\omega^{x_1} =
P(x_1,\cdot) d\sigma$ in \eqref{kernels} and then arguing in a
similar fashion as in the proof of Theorem \ref{T:RHGeneral} but
using the non-degeneracy estimate in $iv)$ of Definition
\ref{D:sADPX} instead of the upper $1$-Ahlfors assumption in
Definition \ref{D:ADPX}. We leave the details to the interested
reader.

\end{proof}

A consequence of Theorem \ref{T:RHGeneral-P} and of Theorem
\ref{T:Doubling Condition} is the following result. We stress that
such result would be trivial if the surface balls would just be the
ordinary Euclidean ones, but this is not the case here. Our surface
balls $\Delta(y,r)$ are the metric ones. Another comment is that
away from the characteristic set the next result would be already
contained in those in \cite{MM}.

\begin{thrm}\label{T:Doubling-s}
Let $D\subset \Rn$ be a $\sigma-ADP_X$ domain. There exist $C,\ R_1
> 0$ depending on the $\sigma-ADP_X$ parameters of $D$ such that for every $y\in\partial  D$ and $0< r < R_1$,
\[
\sigma(\Delta(y,2r))\ \leq\ C\ \sigma(\Delta(y,r))\ .
\]
\end{thrm}

\begin{proof}[\textbf{Proof}]
Applying Theorem \ref{T:RHGeneral-P} with $p=2$, we find
\begin{align*}
\frac{1}{\sigma(\Delta(x_o,r))} \int_{\Delta(x_o,r)} P(x_1,y)^2\
d\sigma(y) &\ \leq\
\left(\frac{C}{\sigma(\Delta(x_o,r))}\int_{\Delta(x_o,r)} P(x_1,y)\
d\sigma(y)\right)^2
\\
&\ =\ C\
\left(\frac{\omega^{x_1}(\Delta(x_o,r))}{\sigma(\Delta(x_o,r))}\right)^2\
.
\end{align*}

This gives
\begin{align*}
\sigma(\Delta(x_o,2r))\ &  \leq \
C\frac{\omega^{x_1}(\Delta(x_o,2r))^2}{\int_{\Delta(x_o,2r)}P(x_1,y)^2\,d\sigma(y)}
\quad\text{(by Theorem \ref{T:Doubling Condition}) } \\
& \leq C
\frac{\omega^{x_1}(\Delta(x_o,r))^2}{\int_{\Delta(x_o,r)}P(x_1,y)^2\,d\sigma(y)}
\ \leq \ C \frac{\left(\int_{\Delta(x_o,r)}P(x_1,y)\,d\sigma(y)\right)^2}{\int_{\Delta(x_o,r)}P(x_1,y)^2\,d\sigma(y)} \\
&\leq \
C\frac{\left(\int_{\Delta(x_o,r)}P(x_1,y)^2\,d\sigma(y)\right)\left(\int_{\Delta(x_o,r)}d\sigma(y)\right)}{\int_{\Delta(x_o,r)}P(x_1,y)^2\,d\sigma(y)}
\ = \ C\ \sigma(\Delta(x_o,r))\ .
\end{align*}

\end{proof}

Our final goal in this section is to study the Dirichlet problem for
sub-Laplacians when the boundary data are in  $L^p$ with respect to
either the measure $\sigma_X$ or the surface measure $\sigma$. We
thus turn to the

\begin{proof}[\textbf{Proof of Theorem \ref{T:Dirichlet}}]

The first step in the proof consists of showing that functions $f\in
L^p(\partial D,d\sigma_X)$ are resolutive for the Dirichlet problem
\eqref{DP}.  In view of Theorem \ref{T:Brelot} it is enough to show
that $f\in L^1(\partial D,d\omega^{x_1})$ for some fixed $x_1\in D$.
This follows from \eqref{kernels} and Proposition
\ref{P:Positivity}, based on the following estimates
\begin{align*}
\int_{\partial  D} |f(y)| \ d\omega^{x_1}(y) &= \int_{\partial  D}
|f(y)| K(x_1,y)\ d\sigma_X(y)
\\
& \leq \left(\int_{\partial D} |f(y)|^p\,d\sigma_X(y)\right)^\frac{1}{p}\left(\int_{\partial D} K(x_1,y)^{p'}\,d\sigma_X(y)\right)^\frac{1}{p'} \\
&\leq C \left(\int_{\partial D}
|f(y)|^p\,d\sigma_X(y)\right)^\frac{1}{p}\ .
\end{align*}

This shows that $L^p(\partial D,d\sigma_X) \subset L^1(\p
D,d\omega^{x_1})$ and therefore, in view of Theorem \ref{T:Brelot},
for each $f\in L^p(\partial D,d\sigma_X)$ the generalized solution
solution $H^D_f$ exists and it is represented by
\begin{equation*}
H^D_f(x) = \int_{\partial D} f(y)\,d\omega^x(y)\ .
\end{equation*}

At this point we invoke Theorem \ref{T:Representation} and obtain
for every $y\in \p D$
\begin{equation}\label{6.7}
N_\alpha(H^D_f)(y)\ \leq\ C\ M_{\omega^{x_1}}(f)(y)\ .
\end{equation}

Moreover, $H^D_f$ converges non-tangentially $d\omega^{x_1}$-$a.e.$
to $f$.  By virtue of Theorems \ref{T:RHGeneral} and
\ref{T:RHGeneral-P}, we also have that $H^D_f$ converges
$d\sigma_X$-a.e. to $f$. To conclude the proof, we need to show that
there exists a constant $C$ depending on $1< p < \infty$, $D$ and
$X$ such that
\[
\|N_\alpha(H^D_f)\|_{L^p(\partial D,d\sigma_X)} \leq C
\|f\|_{L^p(\partial D,d\sigma_X)}\ ,
\]
for every $f\in L^p(\partial D,d\sigma_X)$. In order to accomplish
this we start by proving the following intermediate estimate
\begin{equation}\label{6.8}
\|M_{\omega^{x_1}}(f)\|_{L^p(\partial D,d\sigma_X)} \leq C
\|f\|_{L^p(\partial D,d\sigma_X)}, \quad 1<p\leq\infty\ .
\end{equation}

Since $p>1$, choose $\beta$ so that $0<\beta<p$ and fix $x_1\in D$
as in Theorem \ref{T:RHGeneral}. From \eqref{kernels} and the
reverse H\"older inequality in Theorem \ref{T:RHGeneral} we have
\begin{align*}
& \frac{1}{\omega^{x_1}(\Delta(x_o,r))} \int_{\Delta(x_o,r)} f(y)
d\omega^{x_1}(y)
\\
& \leq
\frac{1}{\omega^{x_1}(\Delta(x_o,r))}\left(\int_{\Delta(x_o,r)}|f(y)|^\beta\,d\sigma_X(y)\right)^\frac{1}{\beta}\left(\int_{\Delta(x_o,r)}
K(x_1,y)^{\beta'}
\,d\sigma_X(y)\right)^\frac{1}{\beta'} \\
& \leq C
\frac{\sigma_X(\Delta(x_o,r))^\frac{1}{\beta'}}{\omega^{x_1}(\Delta(x_o,r))}\left(\frac{1}{\sigma_X(\Delta(x_o,r))}
\int_{\Delta(x_o,r)} K(x_1,y)\,d\sigma_X(y)\right)\,
\|f\|_{L^\beta(\Delta(x_o,r), d\sigma_X)} \\
& = C \left(\frac{1}{\sigma_X(\Delta(x_o,r))} \int_{\Delta(x_o,r)}
|f(y)|^\beta\,d\sigma_X(y)\right)^\frac{1}{\beta}.
\end{align*}

If we now fix $y\in\partial D$ and take the supremum on both sides
of the latter inequality by integrating on every surface ball
$\Delta(x_o,r)$ containing $y$, we obtain
\begin{equation}\label{6.9}
M_{\omega^{x_1}} (f)(y)\ \leq\ C\
M_{\sigma_X}(|f|^\beta)(y)^\frac{1}{\beta}\ .
\end{equation}

By the doubling condition \eqref{deltadoubling} in Corollary
\ref{C:ahlfors} we know that the space $(\partial
D,d(x,y),d\sigma_X)$ is a space of homogeneous type. This allows us
to use the results in \cite{CW} and invoke the continuity in
$L^p(\partial D,d\sigma_X)$ of the Hardy-Littlewood maximal function
obtaining
\begin{align*}
\|M_{\omega^{x_1}}\,f\|^p_{L^p(\partial D,d\sigma_X)}\ & \leq\ C\
\|M_{\sigma_X}(|f|^\beta)^\frac{1}{\beta}\|^p_{L^p(\partial
D,d\sigma_X)}
\\
& =\ \int_{\partial D}
M_{\sigma_X}(|f|^\beta)^\frac{p}{\beta}\,d\sigma_X\ \leq\ C\
\int_{\partial D} |f|^p \,d\sigma_X\ =\ C\ \|f\|^p_{L^p(\partial
D,d\sigma_X)}\ ,
\end{align*}
which proves \eqref{6.8}.  The conclusion of the theorem follows at
once from \eqref{6.7} and \eqref{6.8}.

\end{proof}

Finally, we give the

\begin{proof}[\textbf{Proof of Theorem \ref{T:Dirichletsigma}}]
If the domain $D$ is a $\sigma-ADP_X$-domain, instead of a
$ADP_X$-domain, then using Theorem \ref{T:RHGeneral-P} instead of
Theorem \ref{T:RHGeneral} we can establish the solvability of the
Dirichlet problem for boundary data in $L^p$ with respect to the
standard surface measure. Since the proof of the following result is
similar to that of Theorem \ref{T:Dirichlet} (except that one needs
to use the second identity $d\omega^{x_1} = P(x_1,\cdot) d\sigma$ in
\eqref{kernels} and also Theorem \ref{T:Doubling-s}), we leave the
details to the interested reader.

\end{proof}

\section{\textbf{A survey of examples and some open problems
}}\label{S:ex}

In the study of boundary value problems for sub-Laplacians one faces two type of difficulties. On one side there is the elusive nature of the underlying sub-Riemannian geometry which makes most of the classical results hard to establish. On the other hand, any new result requires a detailed analysis of geometrically significant examples, without which the result itself would be devoid of meaning. This task is not easy, the difficulties being mostly related to the presence of characteristic points. In this perspective it becomes important to provide examples of $ADP_X$-domains. In this section we recall
some of the pertinent results from recent literature.

\medskip

\paragraph{\bf Examples of $NTA_X$ domains}

In the classical setting Lipschitz and even $BMO_1$ domains are
$NTA$ domains \cite{JK}. In a Carnot-Carath\'eodory space it is
considerably harder to produce examples of such domains, due to the
presence of characteristic points on the boundary. In \cite{CG} it
was proved that in a Carnot group of step two every $C^{1,1}$ domain
with cylindrical symmetry at characteristic points is $NTA_X$. In
particular, the pseudo-balls in the natural gauge of such groups are
$NTA_X$. This result was subsequently generalized by Monti and
Morbidelli \cite{MM}.

\begin{thrm}\label{T:MM}
In a Carnot group of step $r=2$ every bounded (Euclidean) $C^{1,1}$
domain is $NTA_X$ with respect to the Carnot-Carath\'eodory metric
associated to a system $X$ of generators of the Lie algebra.
\end{thrm}

\medskip

\paragraph{\bf Examples of domains satisfying the uniform outer $X$-ball property.}
The following result provides a general class of domains satisfying
the uniform $X$-ball condition, see \cite{LU} and \cite{CGN2}. We
recall the following definition from \cite{CGN2}. Given a Carnot
group $\bG$, with Lie algebra $\bg$, a set $A\subset \bG$ is called
\emph{convex}, if $\exp^{-1}(A)$ is a convex subset of $\bg$.

\begin{thrm}\label{T:Convex}
Let $\bG$ be a step two Carnot group of Heisenberg type with a given
orthogonal set $X=\{X_1,...,X_m\}$ of generators of its Lie algebra,
and let $D\subset \bG$ be a convex set. For every $R>0$ and $x_o\in
\partial D$ there exists a $X$-ball $B(x_o,R)$ such that \eqref{ob}
is satisfied. From this it follows that every bounded convex subset
of $\bG$ satisfies the uniform outer $X$-ball condition. In
particular, this is true for the gauge balls.
\end{thrm}

We mention explicitly that, thanks to the results in \cite{K}, in
every group of Heisenberg type with an orthogonal system $X$ of
generators of $\bg = V_1 \oplus V_2$, the fundamental solution of
the sub-Laplacian associated with $X$ is given by \[ \Gamma(x,y)\ =\
\frac{C(\bG)}{N(x^{-1} y)^{Q-2}}\ , \] where $Q = dim(V_1) + 2\
dim(V_2)$ is the homogeneous dimension of $\bG$, and \[ N(x,y)\ =\
(|x|^4 + 16 |y|^2)^{1/4}\ , \] is the non-isotropic Kaplan's gauge.
Kaplan's formula for the fundamental solution shows, in particular,
that in a group of Heisenberg type the $X$-balls coincide with the
gauge pseudo-balls (incidentally, in this setting the gauge defines
an actual distance, see \cite{Cy}). As a consequence of this fact,
the exterior $X$-balls in Theorem \ref{T:Convex} can be constructed
explicitly by finding the coordinates of their center through the
solution of a linear system and a second order equation.

\medskip

\paragraph{\bf Ahlfors type estimates  for the perimeter measure.}
Recall that if $D\subset \Rn$ is a standard $C^1$, or even a
Lipschitz domain, then there exist positive constants  $\alpha,
\beta$ and $R_o$ depending only on $n$ and on the Lipschitz
character of $D$, such that for every $x_o\in \p D$, and every
$0<r<R_o$ one has \begin{equation}\label{ld} \alpha\ r^{n-1}\ \leq\
\sigma(\p D \cap B(x_o,r))\ =\ P(D;B(x_o,r))\ \leq\ \beta\ r^{n-1}\
.
\end{equation}

Here, we have denoted by $P(D,B(x_o,r))$ the perimeter of $D$ in
$B(x_o,r)$ according to De Giorgi. Estimates such as \eqref{ld} are
referred to as the $1$-Ahlfors property of surface measure. They
play a pervasive role in Euclidean analysis especially in connection
with geometric measure theory and its applications to the study of
boundary value problems. In what follows we recall some basic
regularity results for the $X$-perimeter measure which generalize
\eqref{ld} and play a central role in the applications of our
results. We have mentioned in the introduction that from the
standpoint of the Carnot-Carath\'eodory geometry, Euclidean
smoothness of a domain is of no significance. Even for $C^{\infty}$
domains one should not, therefore, expect $1$-Ahlfors regularity in
general, see \cite{CG2} for various examples. For this reason we
introduce the notion of type of a boundary point, and recall a
result showing that if a domain possesses such property, then the
corresponding $X$-perimeter satisfies Ahlfors regularity properties
with respect to the metric balls.

Given a system of $C^\infty$ vector fields $X = \{X_1,...,X_m\}$
satisfying \eqref{frc}, consider a bounded $C^1$ domain $D \subset
\Rn$ with an outer normal $\bN$. We say that a point $x_o\in \p D$
is \emph{of type $\leq 2$} if either there exists $j_o\in
\{1,...,m\}$ such that $<X_{j_o}(x_o),\bN(x_o)> \not= 0$ (i.e.,
$x_o$ is non-characteristic, see Definition \ref{D:char}), or there
exist indices $i_o, j_o\in \{1,...,m\}$ such that
$<[X_{i_o},X_{j_o}](x_o),\bN(x_o)> \not= 0$. We say that $D$ is of
type $\leq 2$ if every point $x_o\in \p D$ is of type $\le 2$. We
stress that when the system has rank $r \le 2$, then every $C^1$
domain is automatically of type $\leq 2$. An important instance is
given by a Carnot group of step $r = 2$. In such a group, every
bounded $C^1$ domain is of type $\leq 2$. The following theorem is
 from \cite{CG}.

\begin{thrm}\label{T:MT-pointwise}
Consider a bounded $C^{1,1}$ domain $D \subset \Rn$. For every point
$x_o\in \p D$ of type $\leq 2$ there exist $A = A(D,x_o)>0$ and $R_o
= R_o(D, x_o)>0$, depending continuously on $x_o$, such that for any
$0<r<R_o$ one has
\begin{equation}\label{ea}
\sigma_X(\Delta(x_o,r))\ \leq\
\left(\underset{y\in\Delta(x_o,r)}{max}\;W(y)\right)\,\sigma(\Delta(x_o,r))\
\leq \ A\ \frac{|B_d(x_o,r)|}{r}\ .
\end{equation}
The same conclusion holds if $\p D$ is real analytic in a
neighborhood of $x_o$, regardless of the type of $x_o$.

If $D$ is a bounded $C^{2}$ domain, then for every point $x_o\in \p
D$ of type  $\leq 2$ there exist $A = A(D,x_o)>0$ and $R_o = R_o(D,
x_o)>0$, depending continuously on $x_o$, such that for any
$0<r<R_o$, one has
\begin{equation}\label{eb}
\sigma_X(\Delta(x_o,r))\ \geq\ A^{-1}\ \frac{|B_d(x_o,r)|}{r}\ .
\end{equation}
\end{thrm}

We mention that in Carnot groups of step $r=2$ the upper $1$-Ahlfors
estimate \eqref{ea} was first proved in \cite{DGN}, whereas for
vector fields of rank $r=2$ the lower estimate \eqref{eb} was first
established in \cite{DGN2}. In the setting of H\"ormander vector
fields, upper Ahlfors estimates for the surface measure $\sigma$
away from the characteristic set were first established in
\cite{MontiMorbidelli}. As a consequence of Theorem
\ref{T:MT-pointwise} we obtain the following

\begin{cor}
Let $X= \{X_1,...,X_m\}$ be a set of $C^\infty$ vector fields in
$\Rn$ satisfying H\"ormander's condition with rank two, i.e.  such
that \[ span\{X_1,...,X_m,[X_1,X_2],....,[X_{m-1},X_m]\}\ =\ \Rn\ ,
\]
at every point. For every bounded $C^{1,1}$ domain $D \subset \Rn$
the horizontal perimeter measure $\sigma_X$ is a $1$-Ahlfors
measure. Moreover the stronger estimate \eqref{ea} holds.
\end{cor}

As a consequence of the results listed above we obtain the following
theorem which provides a large class of domains satisfying the
$ADP_X$ or even the stronger $\sigma-ADP_X$ property.

\begin{thrm}\label{T:ADPXex}
Let $\bG$ be a Carnot group of Heisenberg type and denote by $X =
\{X_1,...,X_m\}$ a set of generators of its Lie algebra. Every
$C^\infty$ convex bounded domain $D\subset \bG$ is a $ADP_X$ and
also a $\sigma-ADP_X$ domain. In particular, the gauge balls in
$\bG$ are $ADP_X$ and also $\sigma-ADP_X$ domains. \end{thrm}

To conclude our review of Ahlfors type estimates, we bring up an
interesting connection between $1$-Ahlfors regularity of the
$X$-perimeter $\sigma_X$ and the Dirichlet problem for the
sub-Laplacian, see \cite{CG2}:

\begin{thrm}\label{T:DP}
Let $D$ be a bounded domain in a Carnot group $\bG$. If the
perimeter measure $\sigma_X$ is $1$-Ahlfors regular, then every
$x_o\in \p D$ is regular for the Dirichlet problem.
\end{thrm}

This result, in conjunction with a class of examples for non-regular domain
constructed in \cite{HanHu} yields the following

\begin{cor}\label{tantiesempi}
If $r\ge 3$ and $m_1\ge 3$, or $m_1=2$ and $r\ge 4$, then there exist
 Carnot groups $\bG$ of step $r\in \N$,  with $\text{dim}\ V_1=m_1$,
and bounded, $C^\infty$ domains $D\subset \bG$, whose perimeter
measure $\sigma_X$ is not $1$-Ahlfors regular.
\end{cor}

\paragraph{\bf Beyond Heisenberg type groups.} The above overview shows that, so far, the known examples
of $ADP_X$ domains are relative to group of Heisenberg type. What
happens beyond such groups? For instance, what can be said even for
general Carnot groups of step two? One of the difficulties here is
to find examples of domains satisfying the outer tangent $X$-ball
condition. The explicit construction in Theorem \ref{T:Convex} above
rests on the special structure of a group of Heisenberg type, and an
extension to more general groups appears difficult due to the fact
that, in a general group, the $X$-balls are not explicitly known and
they may be quite different from the gauge balls. In this connection
it would be desirable to replace the uniform outer $X$-ball
condition with a {\sl uniform outer gauge pseudo-ball condition}
(clearly the two definitions agree for groups of Heisenberg type).
It would be quite interesting to know whether for general Carnot
groups a uniform outer gauge pseudo-ball condition would suffice to
establish the boundedness of the horizontal gradient of the Green
function near the characteristic set (this question is open even for
Carnot groups of step two which are not of Heisenberg type!).
Concerning the question of examples we have the following special
results.

\begin{dfn}
Let $\bG$ be  a Carnot group and  denote by $\bg$ its Lie algebra.
We say that a family $\mathcal F$of  smooth open subsets of $\bg$ is
a $T-$family if it satisfies

(i) For any $F\in \mathcal F$, the manifold $\p F$ is diffeomorphic to the
unit sphere in the Lie algebra.

(ii) The family $\mathcal F$ is left-invariant, i.e. for any $x\in \bG$ and
$F\in \mathcal F$ we have $\log(x\exp (F))\in \mathcal F$.
\end{dfn}

If $D\subset \bg$ is a smooth subset and $\mathcal F$ is a
$T-$family, then we say that $D$ is {\it tangent  to $\mathcal F$}
if for every $x\in \p D$ there exists $F\in \mathcal F$ such that
$x\in \p F$ and the tangent hyperplanes to $\p F$ and $\p D$ at $x$
are identical, i.e.  $T_x \p F=T_x \p D$.

\begin{thrm}
Let $\bg$ be the Lie algebra of a   Carnot group of odd dimension.
If $D\subset \bg$ is a smooth open set   and $\mathcal F$ is a $T-$family,
then  $D$ is {\it tangent  to $\mathcal F$}.
\end{thrm}
\begin{proof}[\textbf{Proof}]
In order to avoid using $\exp$ and $\log$ maps for all $x,y\in \bg$ we will denote by
$xy$ the algebra element $\log (\exp x \exp y)$. We will assume that $\bg$ is endowed
with a Euclidean metric, so that notions of orthogonality and projections can be used.
Fix $x_o\in
\partial D$ and choose any element $F\in \mathcal F$.  We will show that there
exists $z\in \bg$ such that the left-translation $zF$ is tangent to $\p D$ at $x_0$.

Let $n = dim(\bG) = dim(\bg)$ be
odd, and denote by $ S^{n-1}$ the unit (Euclidean) sphere of
dimension $n-1$. Define the map $N:\partial F \to S^{n-1}$ as follows:
For each point $x\in
\partial F$ set  $\tilde D=xx_o^{-1}D$ and observe that this is a smooth open set
with $x\in \p  \tilde D\cap \p F$.
Set $$N(x)=\text{ the outer unit normal to the boundary of the  translated set } \p \tilde D \text { at the point }x.$$
 This amounts to left-translating the point $x_o$ to the
point $x$ and considering the unit normal to the translated domain
at that point. The smoothness of $D$ and of the group structure of
$\bG$ implies that $N$ is a smooth vector field in $\p F$.
 In order to prove the theorem we
need to show that for some point $x\in
\partial F$ the vector $N(x)$ is orthogonal to $T_x\partial
F$. In fact in that case the set $F$ would be tangent to the
translated set $\tilde D$ at the point $x_o$, and its left
translation $x_ox^{-1}F$ could be chosen as the element of $\mathcal
F$ tangent to $D$   at the point $x_o$.  Recall that left
translation, being a diffeomorphism, preserves the property of being
tangent. The conclusion comes from the fact that there cannot be any
smooth tangent non vanishing vector field on $\partial F$ since it
is diffeomorphic to  $ S^{n-1}$. Consequently the vector fields
obtained by  projecting $N(x)$ on $T_x\partial F$ must vanish for
some point $x\in
\partial F$.

\end{proof}

\begin{cor}
Let $\bG$ be a Carnot group of step two with odd-dimensional Lie
algebra $\bg$ and  $D\subset \bg$ be a smooth convex  subset. If
$\mathcal F$ is a $T-$family, composed of convex subsets, and
invariant by the transformation $x\to x^{-1}$ then for any $x\in \p
D$ there exists $F\in \mathcal F$ such that $F\subset D^c$, and
$x\in \p F$.
\end{cor}
\begin{proof}[\textbf{Proof}]
In a Carnot group of step two the left translation map is affine and hence preserves convexity. The same holds for the inverse map.
Consequently at any boundary  point $x_0\in \p D$ there will be a convex
 manifold $F\in \mathcal F$ tangent to $D$ at $x_0$. Being $D$ convex as well
 then $D$ and $F$ must either be on the same side or lay at different sides of the common tangent plane $T_{x_0}\p D$. By translating $x_0$ to the origin and  considering either $F$ or $F^{-1}$ we can pick the manifold lying on the opposite side
 of $D$ and hence   disjoint from it.
\end{proof}

Choosing appropriate $T-$families of convex sets we can now prove
our two main results concerning the uniform outer gauge pesudo-ball and $X-$ball
conditions.

\begin{cor}\label{T:Convex1}
Let $\bG$ be a Carnot group of step two with odd-dimensional Lie algebra $\bg$. Given
a convex set $D\subset \bG$, for every $x_o\in
\partial D$ and every $r>0$ there exists a gauge pseudo-ball $B(x_1,r)$ which is tangent to $\p D$ in $x_o$ from the outside, i.e., such that
\eqref{ob} is satisfied. Furthermore, every bounded convex set in
$\bG$ satisfies the uniform outer gauge pseudo-ball condition.
\end{cor}
\begin{proof}[\textbf{Proof}]
If $D$ is smooth then the proof follows from the immediate
observation that the gauge balls are convex sets in the Lie algebra
and are diffeomorphic to $S^{n-1}$ (see for instance \cite{F2}). For
non-smooth convex domains $D$, we consider $x_o\in \partial D$ and a
new domain $\hat{D}$ obtained as the half space including $D$ and
with boundary $T_{x_o}\partial D$. Since $\hat D$ is a smooth convex
domain then we can apply to it the previous theorem and find an
outer tangent gauge ball  at the point $x_o$ with radius $r>0$.
Clearly this ball will also be tangent to the original  domain $D$
at $x_o$, and will be contained entirely in the complement of $D$.
\end{proof}

\begin{cor}
Let $\bG$ be a Carnot group of step two with odd-dimensional Lie algebra $\bg$. If for every $x\in \R^n$
and for $r$ sufficiently small the $X-$balls
$B(x,r)$  are convex, and $B(x^{-1},r)=B(x,r)^{-1}$ then  every bounded convex set in
$\bG$ satisfies the uniform outer $X$-ball condition.
\end{cor}
\begin{proof}[\textbf{Proof}]
We need only to show that the family of balls $B(x,r)$ form a $T-$family. In \cite{DG2}
it is shown that $X-$balls are starlike with respect to the family of homogeneous dilations in the Carnot group. In particular, one has the estimate
 $$\langle \nabla \Gamma (\cdot, x), Z\rangle >0$$  on $\p B(x,r)$
where we have denoted by $Z$ the generator of the homogeneous dilations. This inequality, coupled
with H\"ormander's hypoellipticity result,
implies that $\p B(x,r)$ is a smooth manifold, while the starlike property immediately
implies that $\p B(x,r)$ is diffeomorphic to the unit ball. \end{proof}

We recall from the classical paper of Folland  \cite{F2} that in a
Carnot group the fundamental solution of the sub-Laplacian is a function
$\Gamma(x,y)=\Gamma(y^{-1}x)$ and $\Gamma(x^{-1})=\Gamma^t(x)$,
where $\Gamma^t$ is the fundamental solution of the transpose of the
sub-Laplacian $\mathcal L$. However, a sub-Laplacian on a Carnot
group is self-adjoint, hence $\mathcal L^*=-\mathcal L$ and
$\Gamma(x)=\Gamma(x^{-1})$.
  Let us denote by $|| \cdot ||$ the group gauge, if we assume that for all $x,y\in \bG$  one has
  $\Gamma(xy^{-1})=\Gamma(yx^{-1})$ (this happens for instance if  $\Gamma(x)=\Gamma(||x||)$), and set  $B(x,r)=\{y|\  \Gamma(y^{-1}x)>c\}$ then $$B(x,r)^{-1}=\{y^{-1}|\  \Gamma(x^{-1}y)>c\}=B(x^{-1},r).$$

We conclude by explicitly noting  that a serious obstruction to extending the previous
results to Carnot groups of higher step consists in the fact
that, unlike in the step two case, the group left-translation may
not preserve the convexity of the sets.

\paragraph{\bf Beyond linear equations.}
Another interesting direction of investigation for the subelliptic
Dirichlet problem is provided by the study of solutions to the
nonlinear equations which arise in connection with the case $p
\not=2$ of the Folland-Stein Sobolev embedding. In this direction a
first step has been recently taken in \cite{GNg} where, among other
results, Theorem \ref{T:Green} has been  extended to the Green
function of the nonlinear equation
\begin{equation}\label{pharm}
\mathcal L_p u = \sum_{j=1}^{2n} X_j(|Xu|^{p-2} X_j u)\ =\ 0\ ,
\end{equation}
in the Heisenberg group $\Hn$. Here is the relevant result.

\begin{thrm}\label{T:pimproved}
Let $D\subset \Hn$ be a bounded domain satisfying the uniform outer
$X$-ball condition. Given $1<p\leq Q$, let $G_{D,p}$ denote the
Green function associated with \eqref{pharm} and $D$. Denote by $g =
(z,t), g' = (z',t')\in \Hn$.
\begin{itemize}
\item[(i)] If $1<p<Q$ there exists a constant $C = C(\bG,D,p)>0$ such that
\[
G_{D,p}(g',g)\ \leq\ C\
\left(\frac{d(g,g')}{|B(g,d(g,g'))|}\right)^{1/(p-1)} d(g',\p D)\ ,\
\ g,g'\in D\ ,\ g'\not=g\ .
\]
\item[(ii)] If $p=Q$, then there exists $C = C(\bG,D)>0$ such that
\[
G_{D,p}(g',g)\ \leq\ C\ \log
\left(\frac{\text{diam}(D)}{d(g,g')}\right)\ \frac{d(g',\p
D)}{d(g,g')}\ ,\ \ g,g'\in D\ ,\ g'\not=g\ .
\]
\end{itemize}
\end{thrm}

One might naturally wonder about results such as Theorem \ref{T:XG}
in this setting. However, before addressing this question one has to
understand the fundamental open question of the interior local
bounds of the horizontal gradient of a solution to \eqref{pharm}.
For recent progress in this direction see the paper \cite{MZZ}.

\end{document}